\documentclass[a4paper,9pt]{amsart}

\usepackage{amssymb}
\usepackage{amsthm,amsmath}
\usepackage{natbib}

\usepackage{enumerate}

\usepackage[all]{xy}

\usepackage{algorithm}
\usepackage{algorithmic}

\newtheorem{theorem}{Theorem}[section]
\newtheorem{proposition}[theorem]{Proposition}
\newtheorem{corollary}[theorem]{Corollary}
\newtheorem{lemma}[theorem]{Lemma}

\newtheorem*{theorem*}{Theorem}
\newtheorem*{proposition*}{Proposition}
\newtheorem*{corollary*}{Corollary}
\newtheorem*{lemma*}{Lemma}
\theoremstyle{definition}
\newtheorem{definition}[theorem]{Definition}
\newtheorem{punto}[theorem]{}
\newtheorem{example}[theorem]{Example}
\theoremstyle{remark}
\newtheorem{remark}[theorem]{Remark}
\newtheorem*{remark*}{Remark}

\newtheorem*{definition*}{Definition}

\newcommand{\lc}{\mathrm{lc}}
\newcommand{\llcm}[2]{[#1,#2]_{\ell}}
\newcommand{\rgcd}[2]{(#1,#2)_{r}}
\newcommand{\lquot}[2]{\operatorname{lquo}(#1,#2)}
\newcommand{\lrem}[2]{\operatorname{lrem}(#1,#2)}
\newcommand{\rend}[2]{\operatorname{End}(_{#1}{#2})}
\newcommand{\assign}{\gets}
\newcommand{\lquotrem}[2]{\operatorname{lquo\_rem}(#1,#2)}
\newcommand{\Ann}{\operatorname{Ann}}
\newcommand{\ffield}[1][]{\ensuremath{\mathbb{F}_{#1}}}
\newcommand{\Cset}{\ensuremath{\mathbb{C}}}

\newcommand{\bigO}{\mathcal{O}}
\newcommand{\Aut}[2][{\ffield[q]}]{\operatorname{Aut}_{#1}(#2)}
\newcommand{\Der}[3][{\ffield[q]}]{\operatorname{Der}_{#1}^{#3}(#2)}

\begin{document}

\title{Computing the bound of an Ore polynomial. Applications to factorization.}

\author{Jos\'{e} G\'{o}mez-Torrecillas}
\address{Department of Algebra and CITIC, University of Granada}
\email{gomezj@ugr.es}
\author{F. J. Lobillo}
\address{Department of Algebra and CITIC, University of Granada}\email{jlobillo@ugr.es}
\author{Gabriel Navarro}
\address{Department of Computer Science and Artificial Intelligence, and CITIC, University of Granada}
\email{gnavarro@ugr.es}

\begin{abstract}
We develop two algorithms for computing a bound of an Ore polynomial over a skew field, under mild conditions. As an application, we state a criterion for deciding whether a bounded Ore polynomial is irreducible, and we discuss a factorization algorithm. The asymptotic time complexity in the degree of the given Ore polynomial is studied.
\end{abstract}

\thanks{Research partially supported by grants MTM2013-41992-P and TIN2013-41990-R from the Ministerio de Econom\'{\i}a y Competitividad  of the Spanish Government and from FEDER.}


\maketitle

\section{Introduction}

Let $R = D[X;\sigma,\delta]$ be an Ore extension of a skew field  $D$ (i.e. a possibly noncommutative field), where $\sigma$ is an automorphism of $D$, and $\delta$ is a $\sigma$-derivation, see for example \cite{Ore:1933} for the details on this construction. This is the best known example of a noncommutative principal ideal domain (PID for short). Every \(f \in R\) has a bound $f^*$. This bound is defined (up to multiplication by nonzero elements in $D$) as a polynomial $f^* \in R$ such $Rf^* = f^*R$ is the largest twosided ideal contained in $Rf$ (or, equivalently, in $fR$, see \cite[Chapter 3]{Jacobson:1943}). An alternative description characterizes $f^*$ as a twosided multiple of $f$ of least degree. The bounds play a prominent role in the structure of cyclic modules. Concretely, since $Rf^*$ is the annihilator of the left $R$--module $R/Rf$, the lattice of submodules of \(R/Rf\) as a left \(R\)--module is the same than its lattice of submodules as a left \(R/Rf^*\)--module. For instance, a factorization of \(f\) as a product of irreducibles is equivalent to provide a composition series of \(R/Rf\), hence it can be studied modulo \(f^*\). Of course this can be exploited whenever $f$ is bounded, that is, $f^* \neq 0$, and $R/Rf^*$ becomes then an Artinian ring. If $R$ is finitely generated as a module over its center, then every nonzero $f$ is bounded. 

Concerning the factorization of Ore polynomials, there is a lot of previous research for some particular cases. In the outer differential case, i.e. \(\sigma\) is the identity and \(\delta\) is not inner, in characteristic zero, the corresponding ring of Ore polynomials is simple (see e.g. \cite[Proposition 2.1]{Goodearl/Warfield:2004}), hence any bound is trivial and other different techniques have to be used to find factors of differential operators. Skew polynomials over finite fields and Ore polynomials over rational functions on finite fields are considered, in relation to the problem of the effective factorization, by Giesbrecht \cite{Giesbrecht:1998}, Caruso and Le Borgne \cite{Caruso/LeBorgne:2017}, and Giesbrecht and Zhang \cite{Giesbrecht/Zhang:2003}, respectively. In \cite{Giesbrecht:1998} a factorization algorithm in \(R=\ffield[][X;\sigma]\) is provided, where $\ffield$ is a finite field. The technique of factorization is based on the results of Ronyai \cite{Ronyai:1987} for finding zero divisors in finite-dimensional algebras over finite fields. Concretely, for $f \in R$ of positive degree, the so-called eigenring $\mathcal{E}(Rf)$ (see \cite[\S 0.4]{Cohn:1971}) is constructed.  Namely, the eigenring is the quotient by $Rf$ of the largest subalgebra of $R$ making $Rf$ a twosided ideal.  Since $\mathcal{E}(Rf)$ is isomorphic to $\mathrm{End}_R(R/Rf)$, the ring of endomorphisms of the left $R$-module $R/Rf$, any non-trivial zero divisor in $\mathcal{E}(Rf)$ provides a non-trivial decomposition of $f$ (see \cite[pp 43-45]{Gomez:2014} for an abstraction of these methods). Such a zero divisor can be computed using the corresponding algorithm in \cite{Ronyai:1987}, since $\mathcal{E}(Rf)$ is a finite-dimensional algebra over the subfield of invariants $K=\ffield[]^\sigma$. Giesbrecht provides  a faster refinement in \cite[\S 5]{Giesbrecht:1998}. The same scheme of factorization based in the eigenring is used in \cite{Giesbrecht/Zhang:2003} when \(R = \ffield[](t)[X;\sigma,\delta]\), by using \cite{Ivanyos/Ronyai/Szanto:1994} instead of \cite{Ronyai:1987}. Unfortunately, the algorithm proposed in \cite{Giesbrecht/Zhang:2003} does not work when $\mathcal{E}(Rf)$ is a simple algebra, so the problem of finding a factorization algorithm for Ore polynomials over \(\ffield[](t)\) remains open. 

The approach in \cite{Caruso/LeBorgne:2017} for factoring \( f \in   R=\ffield[][X;\sigma]\) is based on the computation  of the norm $\mathcal{N}(f)$, a left multiple of $f$ in the center of $R$, and thereof, a multiple of $f^*$.  The  polynomial $f$ is irreducible if and only if $\mathcal{N}(f)$ is irreducible in the center. Furthermore, theoretically, a factorization of $\mathcal{N}(f)$ provides a factorization of $f$. The key-point lies in the effective computation of such factors. When the norm is a power of a irreducible polynomial in the center, i.e. when any bound is irreducible, a probabilistic method is required in order to compute a proper factor. Unfortunately, this method depends heavily on that the base field is finite (actually, on the trivialization of the Brauer group of a finite field) for having a positive probability of success, so it hardly can be translatable to other contexts.

The primary aim of this paper is to give an algorithm  for computing a bound of a bounded Ore polynomial in $R = D[X;\sigma,\delta]$ whenever the input data $D$, $\sigma$ and $\delta$ are effective and computable. As an application, we analyze the reduction of the factorization of Ore polynomials to the commutative factorization in the center $C$ of $R$. This is somehow the underlying idea of the aforementioned papers \cite{Giesbrecht:1998, Caruso/LeBorgne:2017}. Nevertheless, here, the computation of a bound  of a polynomial becomes the key-tool enabling such reduction. Actually, the role of the bound in the structure of bounded modules over Ore polynomial rings (see \cite[Ch. 3]{Jacobson:1943}) yields a generality in our approach that allow us  to cover effectively examples whose factorization theory has never managed, as far as we know, in the literature. Our methods also apply to the examples covered by \cite{Giesbrecht:1998,Giesbrecht/Zhang:2003,Caruso/LeBorgne:2017}, giving alternative algorithms of factorization for them.
Obviously, the techniques give satisfactory results whenever, firstly, the computation of a bound is effective and efficient. We provide very simple algorithms under the assumption of finiteness over the center of $R$ as a module in one case, or as an algebra, in the other. In the latter case, the algorithm runs even when the user does not know explicitly the center. Since we work under the requisite that the center is ``big enough'', our proposal is not suitable, for instance, for differential polynomial rings over fields of characteristic zero. However, skew polynomials rings with large centers have been successfully applied to design new linear codes. See e.g. \cite{Boucher/Ulmer:2009, GLN2016, Boulagouaz/Leroy:2013, Boucher/Ulmer:2014}.

We briefly outline the paper. In Section \ref{bounds}, we fix the notation concerning Ore polynomials and recall some classical and well known definitions and results about them. We provide Algorithm \ref{alg:bound} and Algorithm \ref{alg:bound2} for computing a bound $f^*$ of a bounded polynomial $f$ by an iterated use of the Euclidean division algorithm. The first one runs whenever  a finite set of generators of $R$ as a module over its center $C$ is available, whilst the latter works when we only explicitly know a finite set of generators of $R$ as an algebra over $C$, even if $C$ is not known. In Section \ref{complexity} we calculate the theoretical efficiency  of both algorithms. The computations are based in the knowledge of an upper bound of the degree of $f^*$, when $R$ is free of known finite rank as a $C$--module (Theorem \ref{cotabound}).  

Section \ref{irred} is devoted to treat the problem of the irreducibility of bounded Ore polynomials. Although many of the results are stated under more general conditions, let us assume to ease the description that $R$ is free of finite rank $r$ over its center $C$, and that this last is a (commutative) polynomial ring over a subfield of $D$.  Let $f \in R$ with $f^* \in C$. If $f$ is irreducible, then $f^*$ is irreducible over $C$, but the converse is not true. In Proposition \ref{boundmaxdegree} we give an easy sufficient condition for the converse. The general case is discussed and we provide a criterium for deciding whether or not an Ore polynomial is irreducible by checking if a related simple algebra, built from $R/Rf^*$, is a skew field, see Proposition \ref{idempotenteirreducible}. Section \ref{factoriz} concerns the use of $f^*$ for factoring out $f$ in $R$. We prove in Proposition \ref{factordenormalfactor} that any non-trivial factorization of $f^*$ in $C$ provides a non-trivial factorization of $f$ in $R$. We obtain thus a ``rough decomposition'' $f = g_1 \cdots g_n$, where each $g_i^*$ is irreducible over $C$ (Proposition \ref{roughdecomposition}). The complete factorization of each $g_i$ requires the computation of zero divisors in simple algebras over $C/Cg_i^*$, see Proposition \ref{simpledecomposition}. Unfortunately, this step depends heavily on the field $C/Cg_i^*$. It is not expected to find a factorization algorithm under so general conditions. The cases of Ore extensions of finite fields and rational functions over finite fields are considered and compared with \cite{Giesbrecht:1998, Giesbrecht/Zhang:2003, Caruso/LeBorgne:2017} in Section \ref{otherfact}. In particular, we show that algorithm \texttt{Factorization} in \cite{Giesbrecht/Zhang:2003}  contains two gaps. We propose a solution to one of them, but the other is related to the effective computation of zero divisors of a finite-dimensional simple algebra over a rational function field. So the problem of factoring out Ore polynomials over \(\ffield[q](t)\) is still open until an algorithm for computing zero divisors of simple algebras is available, see Subsection \ref{ratfunc}.

All along the paper, the theory is illustrated by numerous examples. These have been implemented and computed with the aid of mathematical software \textsf{Sage} \cite{sage}.

\section{Computing a bound in Ore polynomials}\label{basics}\label{bounds}

We begin this section fixing notation and recalling some basic facts about Ore extensions and factorization of Ore polynomials over a skew field $D$.  The original construction, as well as the first investigation on the arithmetic and factorization theory of these noncommutative polynomials in full generality, is due to O. Ore \cite{Ore:1933}. Ore polynomials are elements of an associative ring with unit $D[X;\sigma,\delta]$, where $\sigma : D \to D$ is a ring automorphism, and $\delta: D \to D$ is a $\sigma$--derivation, that is, 
\[
\delta (a+b) = \delta(a) + \delta(b) \quad \text{and} \quad \delta(ab) = \sigma(a)\delta(b) + \delta(a)b
\] 
for any $a,b \in D$. The construction of $R = D[X;\sigma,\delta]$ goes as follows: $R$ is a left $D$--vector space on the basis $\{ X^n : n \geq 0 \}$,  
the sum of polynomials is as usual, the product on $R$ is extended recursively from the rules $X^nX^m = X^{n+m}$, for $m, n \in \mathbb{N}$, and $Xa = \sigma(a)X + \delta(a)$ for $a \in D$.

There are two special classes of Ore polynomials. If $\delta = 0$, it is usually written $R = D[x;\sigma]$, and if $\sigma$ is the identity, it is omitted, and we denote $R = D[x;\delta]$. In many situations there is a reduction to one of these special cases, see \cite[Theorem 3.1, pp. 295]{Cohn:1971}.

Let $R = D[X;\sigma,\delta]$. The degree $\deg f$ of a nonzero left polynomial $f \in R$, as well as its leading coefficient $\lc(f) \in D$,  are defined in the usual way. 
We write $\deg 0 = - \infty$, with the usual conventions for this symbol, and $\lc (0) = 0$. 

As shown in \cite[\S 2]{Ore:1933}, the ring $R$ is a left and right Euclidean domain, so it has both left and right division algorithms. 
The \emph{left remainder} and \emph{left quotient} of the left division of $f$ by $g$ are denoted by $\lrem{f}{g}$, $\lquot{f}{g}$ and $(\lquot{f}{g},\lrem{f}{g}) = \lquotrem{f}{g}$. See also \cite[Algorithm 2]{Bueso/Gomez/Verschoren:2003} for a more modern description. 

As a consequence of the division algorithms, $R$ is a noncommutative principal ideal domain (PID), that is, every left and every right ideal is principal (see \cite[Chapter 3]{Jacobson:1943} for details on these rings). Twosided ideals of $R$ are of the form $R\alpha = \alpha R$ for some \emph{normal} or \emph{twosided} polynomial $\alpha$ \cite[pp. 37]{Jacobson:1943}. 

Given $f, g \in R$, we use the notation $g \mid_r f$ to say that $g$ is a \emph{right divisor} (or \emph{right factor}) of $f$, or that $f$ is a \emph{left multiple} of $g$, i.e. $Rf \subseteq Rg$. The right greatest common divisor is denoted by $\rgcd{f}{g}$ and the left least common multiple by $\llcm{f}{g}$, both defined as usual by $Rf + Rg = R\rgcd{f}{g}$ and $Rf \cap Rg = R\llcm{f}{g}$. They can be computed by using the appropriate version of the extended Euclidean algorithm (see, for instance, \cite[Theorem 4]{Ore:1933} and \cite[\S 1.4]{Bueso/Gomez/Verschoren:2003}). The associativity of the sum and intersection of left ideals allows to extend the definition and computation of right greatest common divisors and left least common multiples to any finite set of polynomials in \(R\). 

The factorization theory in Ore polynomials comes from the pioneering paper \cite[Chapter II]{Ore:1933}. It is proven in \cite[Chapter II, Theorem 1]{Ore:1933} that every non-unit polynomial has a representation as a product of irreducible factors, and this representation is unique up to similarity, where $f, g \in R$ are said to be \emph{similar}, $f \sim g$, if there is an isomorphism of left $R$-modules $R/Rf \cong R/Rg$ or, equivalently, of right $R$--modules $R/fR \cong R/gR$.

There is also a factorization theory for twosided polynomials. This theory mimics that of the commutative one as proven in \cite{Jacobson:1943}. 

Given a twosided polynomial of positive degree $\pi \in R$, the ideal $R \pi$ is maximal as a twosided ideal if and only if $R/R \pi$ is a simple Artinian ring. In this case $\pi$ is said to be \emph{prime twosided polynomial}. Therefore, \cite[Theorem 9, p. 38]{Jacobson:1943} may be rephrased by saying that every nonzero twosided polynomial $\alpha \in R$ factorizes uniquely, up to multiplication by nonzero elements of $D$ and reordering, as a product $\alpha = \pi_1 \cdots \pi_r$, where $\pi_1, \dots, \pi_r$ are prime twosided polynomials in $R$.  Note that a twosided polynomial $\pi$ is irreducible if and only if $R/R\pi$ is a skew field. Thus, any irreducible twosided polynomial is prime, but the converse fails in general. 

Let $M$ be a left $R$--module and $S \subseteq M$ any subset.   The annihilator of $S$  is defined as
\begin{displaymath}
\Ann_R(S) = \{ f \in R : fm = 0, \forall m \in S \},
\end{displaymath}
and it is a left ideal of $R$. When $N \subseteq M$ is a submodule then $\Ann_R(N)$ is  a twosided ideal.

\begin{definition}[{\cite[p. 38]{Jacobson:1943}} and {\cite[pp. 227]{Cohn:1971}}]
\emph{A bound} of $f \in R$ is a twosided polynomial $f^* \in R$ such that $Rf^* = f^*R$ is the largest twosided ideal contained in $Rf$ or, equivalently,  
\[
Rf^* = \Ann_R(R/Rf).
\] 
By \cite[Theorem 11, p. 39]{Jacobson:1943},  a  bound $f^*$ of $f$ could be equally defined by using the right ideal $fR$. If $f^* \neq 0$, then $f$ is said to be \emph{bounded}. Obviously, $f^*$ is determined up to multiplication (say on the left) by nonzero elements of $D$. For a more general perspective of bounded elements in 2-firs see \cite[\S 6.5]{Cohn:1971} and \cite{Leroy/Ozturk:2011}.
\end{definition}

\begin{remark}\label{latticefstar}
The left $R$--module $R/Rf$ is a left $R/Rf^*$--module in the obvious way, and, what is more, the lattice of left $R$--submodules of $R/Rf$ is, precisely, the lattice of its left  $R/Rf^*$--submodules.  In particular, we get that $f$ is irreducible if and only if $R/Rf$ is simple as a left $R/Rf^*$--module.  When $f^* \neq 0$, the ring $R/Rf^*$ is finite-dimensional as a left vector space over the  skew field  $D$, and, therefore, it is Artinian. In fact, it is a finite-dimensional algebra over a suitable commutative field in a wide class of examples (including the examples described in Remark \ref{sinderivacion}).
If $f$ is irreducible and bounded, then, by \cite[Theorem 13, p. 40]{Jacobson:1943}, $Rf^*$  is a prime twosided ideal and $R/Rf^*$ is a simple Artinian ring. 
We recommend Pierce's book \cite{Pierce:1982} for readers non familiar with noncommutative associative algebras.
\end{remark}

We propose two algorithms for computing a bound for a given bounded polynomial. The first one runs when $R$ is finitely generated as a module over its center {$C$, and a finite system of generators is known. The second one can be used when $R$ is finitely generated as an algebra over its center. Even if it is known that $R$ is finitely generated as a module over $C$ but no set of generators is available,  or even $C$ is itself unknown, see for instance Example \ref{ex:cyclotomic} or Section \ref{ratfunc}. Since both of them make use of the computation of annihilators of elements in $R/Rf$, we shall need the following lemma.

\begin{lemma}\label{annihilators}
Let $h, f \in R$ be nonzero Ore polynomials with $\deg h < \deg f$, and $f_h \in R$ such that $\llcm{f}{h} = f_h h$. Then $\Ann_R(h + Rf) = Rf_h$.  
\end{lemma}

\begin{proof}
If $g \in \Ann_R(h + Rf)$, then $gh \in Rf \cap Rh = R\llcm{f}{h} = Rf_h h$. Hence $gh = rf_h h$ for some $r \in R$. This implies that $g = r f_h \in Rf_h$, since $R$ is a domain. Conversely, if $g \in Rf_h$, then $gh \in Rf_h h = Rf \cap Rh \subseteq Rf$, whence $g \in \Ann_R(h + Rf)$.
\end{proof}

As a direct consequence of Lemma \ref{annihilators}, Algorithm \ref{LEEAAnn} computes a generator $f_h$ of the left ideal $\Ann_R(h + Rf)$ making use of a short version of the extended Euclidean Algorithm in \cite[Chapter I, Theorem 4]{Ore:1933}.

\begin{algorithm}
\caption{Annihilator of an element}\label{LEEAAnn}
\begin{algorithmic}
\REQUIRE $f, h \in D[X;\sigma,\delta]$ with $f \neq 0, h \neq 0$ 
\ENSURE $g \in D[X;\sigma,\delta]$ with $Rg = \Ann_R(h+Rf)$
\STATE $v_1  \assign 1$, $v_0  \assign 0$, $f_0 \assign f$, $f_1 \assign h$
\STATE $(c,r)  \assign \lquotrem{f_{0}}{f_{1}}$, $f_0  \assign f_1$, $f_1  \assign r$
\STATE $v  \assign - c$, $v_0  \assign v_1$, $v_1  \assign v$
\WHILE{$f_1 \neq 0$}
\STATE $(c,r)  \assign \lquotrem{f_{0}}{f_{1}}$, $f_0  \assign f_1$, $f_1  \assign r$
\STATE $v  \assign v_{0} - c v_{1}$, $v_0  \assign v_1$, $v_1  \assign v$
\ENDWHILE
\RETURN $v_1$
\end{algorithmic}
\end{algorithm}

 If $R$ is finitely generated as a module over its center $C$, then it is well known that every nonzero $f \in R$ is bounded.  For instance, making use of the Cayley-Hamilton theorem, it is possible to get a nonzero element $c \in Rf \cap C \subseteq Rf^*$. Although $c$ is a left multiple of $f^*$, some situations would require to compute exactly $f^*$. The following proposition provides a method for computing $f^*$.

\begin{proposition}\label{bound}
Assume that $R$ is generated, as a module over its center, by finitely many polynomials $c_1, \dots, c_r \in R$. Let $f \in R$ be a nonzero polynomial. Then a bound of $f$ is $f^* = \llcm{f,f_{c_1}, \dots}{f_{c_r}}$, where $Rf_{c_i} = \Ann_R(c_i + Rf)$. 
\end{proposition}

\begin{proof}
Let us first prove that $f^* = \llcm{f_{c_1}, \dots}{f_{c_r}}$. Since 
\[
R[f_{c_1}, \dots, f_{c_r}]_{\ell} = Rf_{c_1} \cap \cdots \cap Rf_{c_r} = \Ann_R(c_1 + Rf) \cap \cdots \cap \Ann_R(c_r + Rf)
\]
and $Rf^* = \Ann_R(R/Rf)$ we get that $Rf^* \subseteq R[f_{c_1}, \dots, f_{c_r}]_{\ell}$. 

For the other inclusion, observe that, given $g \in \Ann_R(c_1 + Rf) \cap \cdots \cap \Ann_R(c_r + Rf)$, and $h \in R$, we may write $h = \sum_{i= 1}^r h_i c_i$, where $h_1, \dots, h_r$ belong to the center of $R$. Thus, 
\[
g (h + Rf) = \sum_{i=1}^r g h_i c_i + Rf = \sum_{i=1}^r h_i g c_i + Rf = 0 + Rf
\]
since $g c_i + Rf = 0 + Rf$ for all $1 \leq i \leq r$.

Finally, since $f^*$ is a left multiple of $f$, we get $f^* = \llcm{f,f_{c_1}, \dots}{f_{c_r}}$.
\end{proof}

The correctness of Algorithm \ref{alg:bound} is a direct consequence of Proposition \ref{bound}. Since $Rf^* = R\llcm{f_{c_1}, \dots}{f_{c_r}}$, $g$ could been initialized to $f_{c_1}$ instead of $f$ in Algorithm \ref{alg:bound}. The initialization to $f$ has the advantage that, if $\llcm{f,f_{c_1},\dots}{f_{c_i}}$ is a twosided polynomial for $i < r$, no further pass through the \textbf{while} loop is needed, since $f^* = \llcm{f,f_{c_1},\dots}{f_{c_i}}$ becomes a bound of $f$. This behavior is illustrated by Example \ref{ex:quaternions}. 

Thus, if a fast criterion for checking whether a given polynomial is twosided is available, it can be added to the condition in the \textbf{while} loop in Algorithm \ref{alg:bound} to make it faster. Such a criterion is given by Theorem \ref{basecota}.(a), if the number $s$ of known generators of $R$ as an algebra over its center is small.

\begin{algorithm}
\caption{Computation of a bound I}\label{alg:bound}
\begin{algorithmic}
\REQUIRE $f \in R = D[X;\sigma,\delta]$ with $f \neq 0$. A finite set of generators $\{ c_1, \dots, c_r \}$ of $R$, as a module over its center.
\ENSURE $f^* \in R$ such that $\Ann_R(R/Rf) = Rf^* = f^*R$.
\STATE $g \assign f$, $i \assign 1$
\WHILE{$i \leq r$} 
\STATE compute $f_{c_i}$ such that $\Ann_R(c_i + Rf) = Rf_{c_i}$
\STATE $g \assign \llcm{f_{c_i}}{g}$
\STATE $i  \assign i+1$
\ENDWHILE
\RETURN $g$
\end{algorithmic}
\end{algorithm}

The idea of making use of a finite set of generators of $R$ as an algebra (instead of as a module) over its center $C$ is the basis for the second algorithm for computing a bound, Algorithm \ref{alg:bound2}. Even if it is known that $R$ is finitely generated, as a module, over its center, the number of generators $s$ as a $C$--algebra is expected to be small compared with the number $r$ of generators of $R$ as a $C$--module. For instance, if  $R = D[X;\sigma]$ with $D = K(a)$ is a simple field  extension of $K = D^\sigma$, $\sigma$ has finite order $\mu$, then $r = \mu^2$ and $s = 2$. So, if the rank of $R$ as a $C$-module is  large enough, it may be convenient to make use of Algorithm \ref{alg:bound2} for computing a bound, see Section \ref{complexity}. 

 Furthermore, Algorithm \ref{alg:bound2} does not require  the computation of a set of generators of $R$ as a $C$--module or even the knowledge of $C$ itself. For instance,
in Section \ref{ratfunc}, Algorithm \ref{InvSubfield} computes the invariant subfield by an automorphism \(\sigma\) of \(\ffield[q](t)\). Since Algorithm \ref{InvSubfield} is exponential in the order of \(\sigma\) in the worst case, the computation of the center of \(\ffield[q](t)[X;\sigma]\) can be very hard. But \(\ffield[q](t)[X;\sigma]\) is generated as an algebra over its center, whoever it is, by \(t\) and \(X\). Algorithm \ref{alg:bound2} provides a way for computing a bound without an explicit knowledge of the center, see Example \ref{ex:cyclotomic} and Example \ref{boundF(t)}.

\begin{theorem}\label{basecota}
Assume that  $R = D[X;\sigma,\delta]$ is generated, as an algebra over its center, by a finite set $\{d_1, \dots, d_s \}$. Let $f \in R$ a nonzero  polynomial. 
\begin{enumerate}[{\indent a)}]
\item $Rf = Rf^*$ if and only if $\lrem{f d_i}{f} = 0$ for all $1 \leq i \leq s$.
\item If $\lrem{f d_{i_0}}{f} \neq 0$ for some $1 \leq i_0 \leq s$, then $Rf^* = R\llcm{f_{d_{i_0}}}{f}^*$, where $f_{d_{i_0}}$ is such that $Rf_{d_{i_0}} = \Ann_R(d_{i_0} + Rf)$, and $R\llcm{f_{d_{i_0}}}{f} \subsetneq Rf$. 
\end{enumerate}
\end{theorem}

\begin{proof}
a) Observe that $Rf = Rf^*$ if and only if $Rf$ is a twosided ideal, and this is equivalent to the inclusion  $fR \subseteq Rf$. Assume $fR \subseteq Rf$. Then, for all $1 \leq i \leq s$, there exists $\widetilde{d_i} \in R$ such that $f d_i = \widetilde{d_i}f$. By the uniqueness of the remainder in a left Euclidean division, this means that $\lrem{f d_i}{f} = 0$ for all $1 \leq i \leq s$. 
Conversely, assume that $\lrem{f d_i}{f} = 0$ for all $1 \leq i \leq s$. This means that $f d_i = \widetilde{d_i}f$ for all $1 \leq i \leq s$, where $\widetilde{d_i} \in R$. Clearly, this implies that $f d_{i_1} \cdots d_{i_m} = \widetilde{d_{i_1}} \cdots \widetilde{d_{i_m}} f$ for every $i_1, \dots, i_m \in \{ 1, \dots, s \}$. Since $d_1, \dots, d_s$ are assumed to be generators of $R$ as an algebra over its center, we see that $fR \subseteq Rf$. 

b) If $\lrem{f d_{i_0}}{f} \neq 0$, then $f \notin \Ann_R(d_{i_0} + Rf) = R f_{d_{i_0}}$. Hence $R \llcm{f_{d_{i_0}}}{f} = Rf_{d_{i_0}} \cap Rf \subsetneq Rf$ and therefore $R\llcm{f_{d_{i_0}}}{f}^* \subseteq Rf^*$. On the other hand, taking annihilators in the canonical injective homomorphism of left $R$--modules
\[
\frac{R}{R\llcm{f_{d_{i_0}}}{f}} = \frac{R}{Rf_{d_{i_0}} \cap Rf} \to \frac{R}{Rf_{d_{i_0}}} \oplus \frac{R}{Rf},
\]
we get 
\[
\Ann_R\left( \frac{R}{Rf} \oplus \frac{R}{Rf_{d_{i_0}}} \right) \subseteq \Ann_R \left( \frac{R}{R\llcm{f_{d_{i_0}}}{f}} \right) = R\llcm{f_{d_{i_0}}}{f}^*.
\]
Moreover, $Rf^* = \Ann_R(R/Rf) \subseteq \Ann_R(d_{i_0} + Rf) = R f_{d_{i_0}}$, so $R f^* \subseteq R f_{d_{i_0}}^*$ and,  therefore, 
\[
R f^* = R f^* \cap R f_{d_{i_0}}^* = \Ann_R \left( \frac{R}{Rf} \right) \cap \Ann_R \left( \frac{R}{Rf_{d_{i_0}}} \right) = \Ann_R \left( \frac{R}{Rf} \oplus \frac{R}{Rf_{d_{i_0}}} \right).
\]
Then $Rf^* \subseteq R\llcm{f_{d_{i_0}}}{f}^*$. This finishes the proof.
\end{proof}

\begin{theorem}\label{thm:bound2}
Assume that  $R = D[X;\sigma,\delta]$ is generated, as an algebra over its center, by a finite set $\{d_1, \dots, d_s \}$. Let $f \in R$ be a bounded  polynomial. Algorithm \ref{alg:bound2} correctly computes a bound $f^*$ of $f$.
\end{theorem}
\begin{proof}
In view of Theorem \ref{basecota}, we must only argue why Algorithm \ref{alg:bound2} terminates. Since $f^* \neq 0$, we know that $R/Rf^*$ is an Artinian ring (and a left $R$--module of finite length, of course). On the other hand, in part b) of Theorem \ref{basecota}, $Rf^* \subseteq R\llcm{f_{d_{i_0}}}{f} \subsetneq Rf$. Therefore, the number of times that the \textbf{while} loop of Algorithm \ref{alg:bound2} runs is bounded by the (finite) length of $R/Rf^*$ as a left $R$-module.  Observe that this length is lower or equal than $\deg f^*$. 
\end{proof}

\begin{algorithm}
\caption{Computation of a bound II}\label{alg:bound2}
\begin{algorithmic}
\REQUIRE $f \in R = D[X;\sigma,\delta]$ with $f^* \neq 0$. A finite set of generators $\{ d_1, \dots, d_s \}$ of $D[X;\sigma,\delta]$ as an algebra over its center
\ENSURE $f^* \in R = D[X;\sigma,\delta]$ such that $\Ann_R(R/Rf) = Rf^* = f^*R$
\STATE $g \assign f$
\STATE $i \assign 1$
\WHILE{$i \leq s$}
\IF{$\lrem{g d_i}{g} = 0$}
\STATE $i \assign i + 1$
\ELSE
\STATE $g \assign \llcm{g_{d_i}}{g}$ with $Rg_{d_i} = \Ann_R(d_i + Rg)$
\STATE $i \assign 1$
\ENDIF
\ENDWHILE
\RETURN $g$
\end{algorithmic}
\end{algorithm}

The running time of Algorithm \ref{alg:bound2} will depend on how many times the left least common multiple $\llcm{g_d}{g}$ has to be computed. In each of these computations the degree of $g$ strictly increases, hence if we had an estimation of the degree of a bound $f^*$ in terms of the degree of the given polynomial $f$, we could deduce an upper bound for the number of left least common multiples to be computed. 
 Theorem \ref{cotabound}  asserts that, under rather general} conditions, the degree of a bound of a given polynomial $f$ can be estimated.  Its proof  will also provide a fast criteria for the irreducibility.  We shall need the following lemmata. 

\begin{lemma}\label{freeness2factor}
Assume $R$ is a finitely generated free module of rank $r$ over its center $C$ and let $\alpha \in C$. Then $R/R\alpha$ is  a  free  module  of rank $r$ over $C/C\alpha$. Moreover, if $\alpha$ is a prime twosided polynomial in $R$, then $C/C\alpha$ is a (commutative) field. 
\end{lemma}

\begin{proof}
Since $R$ is a domain, it is easily checked that $C\alpha = R\alpha \cap C$, whence we have the obvious monomorphism of rings $C/C\alpha \rightarrow R/R\alpha$, which makes $R/R\alpha$ an algebra over $C/C\alpha$. On the other hand, $R \otimes_C(C/C\alpha) \cong R/R\alpha$ as $C/C\alpha$--modules. Since $R$ is free of rank $r$ as a $C$--module, and $R \otimes_C -$ preserves direct sums, it follows that $R/R\alpha$ is free of rank $r$ as a $C/C\alpha$--module. Since $R$ is a finitely generated free $C$--module and $R$ is a (left) Noetherian ring, it follows  from \cite[Corollary 1.1.4]{McConnell/Robson:1987}  that $C$ is Noetherian and, therefore, $C/C\alpha$ is Noetherian. Now,  $C/C\alpha$ is clearly contained in the center $Z$ of $R/R\alpha$. Therefore, $Z$ becomes a $C/C\alpha$--submodule of $R/R\alpha$. Since $C/C\alpha$ is Noetherian, it follows that $Z$ is finitely generated as a $C/C\alpha$--module. Now, if $\alpha$ is a prime twosided  element of $R$, then $R/R\alpha$ is a simple Artinian algebra and, thus, its center $Z$ is a field, and it is a finite extension of $C/C\alpha$. Then $C/C\alpha$ is a field (see, e.g. \cite[Proposition 5.7]{Atiyah/Macdonald:1969}). 
\end{proof}

\begin{lemma} \cite{Jacobson:1943}\label{boundindecomposable}
Let \(f \in R\) be bounded such that \(R/Rf\) is indecomposable. Then there exist a prime twosided polynomial \(\pi \in R\) and irreducible elements \(p_1, \dots, p_l \in R\) such that \(f = p_1 \dots p_l\), \(p_i^* = \pi\) for all \(1 \leq i \leq l\) and \(Rf^* = (R\pi)^l\). 
\end{lemma}

\begin{proof}
By \cite[Theorem 13, p. 40]{Jacobson:1943}, $Rf^* = (R\pi)^e$ for some maximal twosided ideal $R\pi$.  Take a factorization  $f = p_1 \cdots p_l$ into irreducible polynomials $p_i$. For all $1 \leq i \leq l$, $Rf^* \subseteq Rp_i^*$, and $Rp_i^*$ is a maximal twosided ideal by \cite[Theorem 13, p. 40]{Jacobson:1943}. Thus $R\pi = Rp_i^*$ for all $1 \leq i \leq l$. By \cite[Theorem 21, p. 45]{Jacobson:1943}, $l = e$
\end{proof}

\begin{theorem}\label{cotabound}
Assume that $R$ is a finitely generated free module of rank $r$ over its center $C$.  Let  $f \in R$ with bound $f^* \in C$. Then $\deg f^* \leq \sqrt{r} \deg f$. 
\end{theorem}

\begin{proof}
Let us first prove the inequality when $f$ is irreducible. By Remark \ref{latticefstar}, $R/Rf^*$ is a simple Artinian ring. Since $R/Rf$ is simple as a left $R/Rf^*$--module, we have that 
\begin{equation}\label{WAD}
\frac{R}{Rf^*} \cong \frac{R}{Rf} \oplus \overset{(m)}{\cdots} \oplus \frac{R}{Rf}
\end{equation}
as left $R/Rf^*$--modules. 
By  Lemma \ref{freeness2factor}, $C/Cf^*$ is a field, and $R/Rf^*$ is a simple Artinian algebra of dimension $r$ over $C/Cf^*$.  Let $F = \rend{R/Rf^*}{R/Rf} = \rend{R}{R/Rf}$, which is a skew field over $C/Cf^*$ by Schur's Lemma. Now, $R/Rf^*$ is a left vector space over $F$ of dimension $m^2$. Therefore,
\begin{equation}\label{mr}
m^2 = \dim_F R/Rf^* = \frac{\dim_{C/Cf^*}R/Rf^*}{\dim_{C/Cf^*}F} = \frac{r}{\dim_{C/Cf^*}F} \leq r,
\end{equation}
then $m \leq \sqrt{r}$. On the other hand, by \eqref{WAD}, we obtain
\begin{equation}\label{eq:simple}
\deg f^* = \dim_D R/Rf^* = m \dim_D R/Rf = m \deg f \leq \sqrt{r} \deg f.
\end{equation}
For a general polynomial $f$, we prove that $\deg f^* \leq \sqrt{r} \deg f$ by induction on the number of indecomposable direct summands in a Krull-Schmidt decomposition of the left $R$--module $R/Rf$. So, we first assume that $R/Rf$ is indecomposable. 
By Lemma \ref{boundindecomposable} \(f = p_1 \dots p_l\) where \(p_i\) is irreducible, \(p_i^* = \pi\) for a prime twosided polynomial, and \(Rf^* = (R \pi)^l\). We already know that $\deg p_i^* \leq \sqrt{r} \deg p_i$ for each $i = 1, \dots, l$. Therefore, 
\begin{equation}\label{eq:ind}
\sqrt{r} \deg f = \sqrt{r} \sum_{i=1}^l \deg p_i \geq l \deg \pi = \deg f^*.
\end{equation}
Thus, the first step of the induction is done.  Assume now that  $R/Rf$ is not indecomposable.  Then $R/Rf \cong R/Rg \oplus R/Rh$, with $g, h$ non constant bounded polynomials.  Obviously, the number of indecomposable direct summands in a Krull-Schmidt decomposition of $R/Rg$ and $R/Rh$ is strictly smaller than that of $R/Rf$.  Making use of the induction hypothesis, we get
\begin{equation}\label{eq:bound}
\begin{aligned}
 \sqrt{r} \deg f & = \sqrt{r} \dim_D R/Rf \\
 & = \sqrt{r} (\dim_D R/Rg+ \dim_D R/ Rh) \\
 & =  \sqrt{r}(\deg g + \deg h) \geq \deg g^* + \deg h^* \\
 & = \deg g^*h^* \\
 & \geq  \deg f^*,
\end{aligned}
\end{equation}
 where the last inequality  follows  from $Rf^* = Rg^* \cap Rh^* \supseteq Rg^*h^*$. 
\end{proof}

\begin{remark}\label{sinderivacion}
One class of Ore extensions fulfilling the conditions of Theorem \ref{cotabound} is $R= D[X;\sigma]$, where $D$ is a  skew field  which is finite-dimensional over its center $C(D)$,  and $\sigma$ is an automorphism of $D$ of finite order $\mu$ modulo an inner automorphism $\alpha \mapsto u \alpha u^{-1}$ where $u \in D\setminus\{0\}$. Let $D^{\sigma}$ be the invariant skew subfield under $\sigma$. By \cite[Theorem 2.8]{Lam/Leroy:1988} or \cite[Theorem 1.1.22]{Jacobson:1996}, the center of $R$ is $C = K[z]$, with $K=C(D) \cap D^\sigma$ and $z= u^{-1} X^\mu$ (see also \cite{Cauchon:1977}). It turns out that $D$ has finite dimension over $K$ and that $R$ is free of finite rank over $C$. Moreover, a bound $f^*$  of a nonzero polynomial $f \in R$ is of the form $f^* = d \widehat{f} X^m$ for some $m \geq 0$, nonzero $d \in D$, and $\widehat{f} \in C$. Let us analyze those cases in which $m \geq 1$.
\end{remark}

\begin{lemma}\label{quitoXm}
Let \(R = D[X;\sigma]\) and \(f \in R \setminus\{0\}\). If $f^* = d \widehat{f} X^m$ then $f = g X^m$ for some $g \in R$ such that $g^* = d \widehat{f}$.
\end{lemma}

\begin{proof}
Obviously, $f^* = qf$ for some $q \in R$. Assume $m \geq 1$, then $X \mid_r qf$ and hence $X \mid_r q$ or $X \mid_r f$. If $X \mid_r q$ then $q = q_0 X = X q_1$, so $X q_1 f = q f = X^m d' \widehat{f}$ and so $q_1 f = X^{m-1} d' \widehat{f}$, a twosided element. It follows that $f^* \mid X^{m-1} d' \widehat{f}$, but this is impossible since $\deg f^* = m + \deg \widehat{f}$. Therefore $X \mid_r f$ and $f = f' X$. Then $q f' X = qf = f^* = d \widehat{f} X^m$, which implies $qf' = d \widehat{f} X^{m-1}$, and so $(f')^* \mid_r d \widehat{f} X^{m-1}$. Now $(f')^* X = q' f' X = q' f$ for some $q'$, therefore $f^* \mid_r (f')^*X$. If $\deg(f')^* < \deg \widehat{f} + m-1$, then $\deg f^* \leq \deg (f')^* + 1 < \deg \widehat{f} + m - 1 + 1 = \deg f^*$, a contradiction. Hence $(f')^* = d X^{m-1} \widehat{f}$. Repeating this process $m$ times we get $f = g X^m$ and $g^* = d \widehat{f}$ as desired. 
\end{proof}

Observe that Theorem \ref{cotabound} works since there is a bound in the center. When $R$ is under the conditions of Remark \ref{sinderivacion}, we do not need this restriction as the following corollary shows.

\begin{corollary}
Under the conditions of Remark \ref{sinderivacion}, for all $f \in R=D[X;\sigma]$, $\deg f^* \leq \sqrt{r} \deg f$.
\end{corollary}

\begin{proof}
By the previous lemma, if $f^* = d \widehat{f} X^m$ then $f = g X^m$ with $g^* = d \widehat{f}$. By Theorem \ref{cotabound}, $\deg g^* \leq \sqrt{r} \deg g$, so
\begin{displaymath}
\deg f^* = m + \deg g^* \leq m + \sqrt{r} \deg g = m + \sqrt{r} (\deg f - m) \leq \sqrt{r} \deg f,
\end{displaymath}
as desired. 
\end{proof}

\begin{example}\label{ex:quaternions}
Let $D$ be the standard quaternion algebra over the rational field, i.e., $D=\mathbb{Q}\oplus \mathbb{Q}i \oplus \mathbb{Q} j \oplus \mathbb{Q} k$, where $i^2=-1$, $j^2=-1$ and $ij=-ji=k$. Consider the inner automorphism $\sigma:D\rightarrow D$ given by $\sigma(a)=ua  u^{-1}$, where $u=1+i$. Hence, $\sigma$ has order one with respect to an inner automorphism, and its invariant subfield is $D^\sigma=\mathbb{Q}\oplus \mathbb{Q}i$. As observed in Remark \ref{sinderivacion}, the center of the skew polynomial algebra $R=D[X;\sigma]$ is $C(R)=\mathbb{Q}[z]$, where $z=u^{-1}X$. A basis of $R$ over $C$ is given by $\{1,i,j,k\}$ so, in particular, by Theorem \ref{cotabound},   $\deg f^*\leq 2\cdot \deg f$  for any polynomial $f\in R$.

Let us consider the polynomial in $R$,
$$f=iX^2+(k+1)X+j+k.$$
We follow the steps of Algorithm \ref{alg:bound}. By Algorithm \ref{LEEAAnn}, $\Ann_R(i+Rf)=Rf_i$, where
$$f_i=-iX^2 + (-1 + k)X + j + k.$$
Now,  $\llcm{f}{f_i}=f^*$, where
$$f^*=-4X^4 + (4 + 4i)X^3 - 8iX^2 + (8 - 8i)X + 8.$$
This is a bound of $f$. Indeed, by making the change of variable $z= u^{-1}X$, $f^*$ can be written as follows,
$$f^*=\widehat{f}=16z^4 - 16z^3 + 16z^2 + 16z + 8$$
a  polynomial in  $\mathbb{Q}[z]$.
\end{example}

\begin{example}\label{ex:cyclotomic}
Let $D = \mathbb{Q}(\xi)$, where $\xi$ is a primitive 7th root of unit, and $\sigma: D \rightarrow D$ defined by $\sigma(\xi) = \xi^2$.  Let then $R=D[X;\sigma]$ and $f$ be the polynomial
$$f=\left(\xi^{5} - 1\right) X^{3} + \xi X + 3 \xi^{2} - 1.$$
Then, by applying Algorithm \ref{alg:bound2} to the set of generators $\{\xi,X\}$ of $R$ as an algebra over its center, a bound $f^*$ of $f$  is  
$$f^*=X^{9} + \left(-\xi^{4} - \xi^{2} - \xi - 2\right) X^{6} +
\left(-\frac{5}{7} \xi^{4} - \frac{5}{7} \xi^{2} - \frac{5}{7} \xi -
\frac{20}{7}\right) X^{3} + \frac{58}{7} \xi^{4} + \frac{58}{7}
\xi^{2} + \frac{58}{7} \xi - \frac{13}{7}. $$
Observe that we did not need to know the center of the ring. In this case, the order of $\sigma$ is $3$  and $D^\sigma = \mathbb{Q}(\xi^4+\xi^2+\xi)$, so the center of $R$ is $C(R) = D^\sigma[X^3]$. The coefficients of $f^*$  belong  to $D^\sigma$ and,  thus,   $f^* = \widehat{f} \in C(R)$.
\end{example}

\begin{example}\label{finitefields}
Let $\mathbb{F} = \mathbb{F}_{256} = \mathbb{F}_2(a)$, where the minimal polynomial of $a$ is $x^{8} + x^{4} + x^{3} + x^{2} + 1$. Let $\tau$ be the Frobenius automorphism and $\sigma = \tau^2$, i.e., $\sigma(\alpha) = \alpha^{2^2} = \alpha^4$ for all $\alpha \in \mathbb{F}$. Then the invariant subfield $\mathbb{F}^\sigma = \mathbb{F}_4 = \{0,1,b = a^{85}, b + 1 = a^{170}\}$. We set the skew polynomial ring $R = \mathbb{F}[x;\sigma]$, whose center is $C = \mathbb{F}_4[z]$, where $z = x^4$. For brevity, we shall write the elements of $\ffield[256]$, different of 0 and 1, as powers of the primitive element $a$. Let us compute the bound of the polynomial $f\in R$ given by
\begin{displaymath}
f  =  x^{100} + a x^{43} +a^{120} x^{20} + a^{35}x^4+ a^{205}.
\end{displaymath}
By applying Algorithm \ref{alg:bound} or \ref{alg:bound2}, we get
\begin{displaymath}
\begin{split}
f ^* =  & x^{400} + a^{85} x^{320} + x^{304} + a^{170}x^{300}+ x^{240}+a^{85}x^{208}+a^{170}x^{200}+
a^{85}x^{172}+a^{85}x^{160}+\\ & \quad +x^{144}+a^{170}x^{140}+a^{170}x^{128}+x^{120}+a^{85}x^{112}+x^{108}+a^{170}x^{104}+
a^{85}x^{100}+\\ & \quad +x^{80}+a^{170}x^{16}+a^{85},
\end{split}
\end{displaymath}
which, viewed as an element of the center $R$, is given by
\begin{displaymath}
\begin{split}
\widehat{f} =  & z^{100} + b z^{80} + z^{76} + (b+1)z^{75}+ z^{60}+bz^{52}+(b+1)z^{50}+
bz^{43}+bz^{40}+\\ & \quad +z^{36}+(b+1)z^{35}+(b+1)z^{32}+z^{30}+bz^{28}+z^{27}+(b+1)z^{26}+
bz^{25}+\\ & \quad +z^{20}+(b+1)z^{4}+b.
\end{split}
\end{displaymath}

\end{example}

\begin{example}\label{boundF(t)}
Let $\ffield[16] = \ffield[2](a)$ where $a^4 = a + 1$,  and  $D=\ffield[16](t)$, the field of rational functions over $\ffield[16]$.  Consider  the automorphism $\sigma:D \rightarrow D$   defined by $\sigma(t) = a^5t$.  The  order of $\sigma$ is $3$ and the invariant subfield is $D^\sigma = \ffield[16](t^3)$. Let $R = D[X;\sigma]$, with center $C(R) = D^\sigma[X^3]$. Let $f$ be the polynomial
$$f=X^{2} + \left(\frac{1}{t + a}\right) X + a t^{2} + 1$$
Then we can apply Algorithm \ref{alg:bound2} to compute $f^*$, and we obtain that
$$f^*=X^{6} + \left(\frac{(a^{3} + a) t^{3} + a^{2} + a + 1}{a^{2}
t^{3} + a^{2} + a}\right) X^{3} + a^{3} t^{6} + 1.$$
\end{example}

\begin{example}\label{withdelta}
Let us show a small example of a differential polynomial ring. Consider the field of rational functions over $\mathbb{F}_2$ and the derivation $\delta$ given by the usual derivative, i.e. $\delta(f(t))=f'(t)$. Hence, let $R=\mathbb{F}_2(t)[X;\delta]$. Following \cite{Lam/Leroy:1988, Jacobson:1996, Giesbrecht/Zhang:2003}, the center of $R$ is $C(R)=\mathbb{F}_2(t^2)[X^2]$. Let $f=X+t\in R$ and follow the steps of Algorithm \ref{alg:bound}. We firstly need to compute a generator of $\Ann_R(t+Rf)$. By using Algorithm \ref{LEEAAnn}, this annihilator is generated by $\frac{1}{t}X+\frac{t^2+1}{t^2}$, or multiplying by $t^2$, by $f_t=tX+t^2+1$. Then $\llcm{f_t}{f}=tX^2+(t^3+t)=t(X^2+t^2+1)$. So that $f^*=X^2+(t^2+1)\in C(R)$ is a bound of $f$.
\end{example}

\section{Complexity}\label{complexity}

The time complexity in the calculation of a bound depends heavily on the automorphism and derivation defining the Ore polynomial ring $R=D[X;\sigma,\delta]$, as well as the skew field $D$. The analysis of the general case, i.e. \(\sigma\) is not the identity and \(\delta\) is not zero, can be dropped for two reasons. First,  the number of sums, multiplications and applications of $\sigma$ and $\delta$ needed to compute $X^i a$ belongs to $\bigO(2^i)$, so this bound is carried in the remaining algorithms. Second, we observed in Section \ref{bounds} that, in most situations, the general case can be reduced to the cases \(\delta = 0\) or \(\sigma\) being the identity, see \cite[Theorem 3.1, pp. 295]{Cohn:1971}.  

In this section, we analyze the complexity when $R= D[X;\sigma]$, where $D$ is a skew field which is finite-dimensional over its center $C(D)$, and $\sigma$ is an automorphism of $D$ of finite order. The pure derivation case can be done analogously by suitable conditions and by adjusting the complexity of the basic operations, but the cost is greater. In fact, as we will see in Lemma \ref{numops}, the product in the pure automorphism case is quadratic in the degree, but the product in the pure derivation case coincides with the cost of matrix multiplication, see \cite{VanDerHoeven:2002}. So this increase of the cost of the multiplication is carried out. As usual we use $\operatorname{MM}(m)$ to denote the number of basic operations in a field $L$ needed to multiply matrices of size $m \times m$ over $L$. The usual scholar method says that \(\operatorname{MM}(m) \in \bigO(m^3)\), and the most recent paper \cite{LeGall:2014} reduces it to \(\bigO(m^{2.373})\).

Let $D^{\sigma}$ be the invariant  skew subfield under $\sigma$. As seen in Remark \ref{sinderivacion}, $D$ has finite dimension $\mu$ over $K = C(D) \cap D^\sigma$ and $R$ is free of finite rank $r$ over $C$.

The cost of the arithmetic in $D$ is described  in terms of $K$, thus it is quite natural to assume that this cost depends on $\mu$. The second column of Table \ref{tab:genericarithmeticcost} includes labels for upper bounds of the number of basic operations in $K$ to carry out the arithmetic on $D$.

\begin{table}[htbp]
  \centering
  \begin{tabular}{|c|c|c|c|}
    \hline
    & Generic & $K \subseteq \ffield[256]$ & $\ffield[q] \subseteq \ffield[q^\mu]$ \\
    \hline
    \hline
    $a+b$ & $\operatorname{S}(\mu)$ & $1$ & $\mu$ \\ 
    \hline
    $ab$ & $\operatorname{M}(\mu)$ & $1$ & $\mu \log \mu \log \log \mu$ \\ 
    \hline
    $a^{-1}$ & $\operatorname{I}(\mu)$ & $1$ & $\operatorname{M}(\mu) \log \mu$ \\ 
    \hline
    $\sigma^{i}(a)$ & $\operatorname{A}(\mu)$ & $1$ & $\mu \operatorname{M}(\mu) \log \mu$ \\ 
    \hline
  \end{tabular}
  \caption{Cost of the arithmetic of $D$ in terms of the arithmetic of $K$, including Examples \ref{expl:F256} and \ref{expl:finitefield}.}
  \label{tab:genericarithmeticcost}
\end{table}

Why do we provide a common upper bound for all powers of $\sigma$? Since the order of $\sigma$ is finite, we may adopt $\operatorname{A}(\mu)$ the maximum of the upper bounds of the number of basic operations needed to perform each power of $\sigma$. Of course, it is natural to think that the computation of $\sigma^i$ needs $i$ times more the computation of $\sigma$. However, in several examples we are going to deal with, this is not the right way to calculate $\sigma^i$. Moreover, this assumption also helps the computation of the complexity of the extended Euclidean algorithms to compute right greatest common divisors and left least common multiples.

\begin{example}\label{expl:F256}
Let $D = \ffield[256]$ and $K$ be any subfield invariant by $\sigma$. In this case $\ffield[256]$ is small enough to be tabulated. Hence we can save four tables including additions, multiplications, inverses, and the powers of $\sigma$ (up to eight). All basic operations are an access to a table, and we can assume that all of them have the same cost, normalized to $1$.
\end{example}

\begin{example}\label{expl:finitefield}
Let $K = \ffield[q]$ and $D = \ffield[q^\mu]$. The idea behind this example is that $K$ is small enough to be tabulated but $D$ is quite big and its arithmetic has to be algorithmic with respect to $K$. The complexity is inherited from the integer multiprecision arithmetic. Hence it is well known that $\operatorname{S}(\mu) = \mu$, $\operatorname{M}(\mu) = \mu \log \mu \log \log \mu$ with the well known algorithms of Sch\"{o}nhage \& Strassen \cite{Schonhage/Strassen:1971} and Sch\"{o}nhage \cite{Schonhage:1977}, or Cantor \& Kaltofen \cite{Cantor/Kaltofen:1991}, $\operatorname{I}(\mu) = \operatorname{M}(\mu) \log \mu$, and $\operatorname{A}(\mu) = \mu \operatorname{M}(\mu) \log(\mu)$ using an algorithm of von zur Gathen \& Shoup \cite{vonzurGathen/Shoup:1992}. Observe that $\sigma$ is a power of the Frobenius automorphism $\tau(a) = a^p$ where $p$ is the prime factor of $q$. Hence $\sigma^i(a)$ consists in the computation of a suitable (and bounded) power of $a$.
\end{example}

\begin{lemma}\label{numops}
Let $f,g \in D[X;\sigma]$ be nonzero polynomials of degree $n$ and $m$, respectively. Then each of the following operations can be performed with the corresponding number of basic operations:
\begin{itemize}
\item $fg$ with $B_{\textsc{mult}}(n,m) = nm\operatorname{S}(\mu)+(n+1)(m+1)\operatorname{M}(\mu)+n(m+1)\operatorname{A}(\mu)$.
\item $\lquotrem{f}{g}$ with $B_{\textsc{div}}(n,m) = \operatorname{I}(\mu) + \tfrac{1}{2}(n-m+1)(n+m+4)\operatorname{S}(\mu) + (n-m+1)(m+2) \operatorname{M}(\mu) + (n-m)(m+2) \operatorname{A}(\mu)$.
\item $\rgcd{f}{g}$ with $B_{\textsc{rgcd}}(n,m) = (m-d+1) \operatorname{I}(\mu) + \Big( \tfrac{1}{2} (n-m+1)(n+m+4) + (m-d)(m+d+4) \Big) \operatorname{S}(\mu) + \Big( (n-m+1)(m+2) + (m-d)(m+d+3) \Big) \operatorname{M}(\mu) + \Big( (n-m)(m+2) + \tfrac{1}{2} (m-d)(m+d+3) \Big) \operatorname{A}(\mu)$.
\end{itemize}
\end{lemma}

\begin{proof}
The proof is completely analogous to the commutative case, which can be viewed in \cite{vonzurGathen/Gerhard:2003}.
\end{proof}

\begin{proposition}\label{numops:LEEAAnn}
Let $f,h \in D[X;\sigma]$ nonzero such that $\deg f = n > m = \deg h$ and $\deg \rgcd{f}{h} = d$. An upper bound of the number of basic operations needed to apply Algorithm \ref{LEEAAnn} is 
\begin{displaymath}
\begin{split}
B_{\textsc{ann}}(n,m,d) &= (m-d+1) \operatorname{I}(\mu) + \tfrac{1}{2} (n+m-2d+1)(n+m+4) \operatorname{S}(\mu) \\
&\quad + \Big( (n-m+1)(m+2) + (m-d)(2n+4) \Big) \operatorname{M}(\mu) \\
&\quad + \Big( (n-m)(m+2) + (m-d)(n+2) \Big) \operatorname{A}(\mu).
\end{split}
\end{displaymath}
\end{proposition}

\begin{proof}
Algorithm \ref{LEEAAnn} is just the left extended euclidean algorithm but only one Bezout coefficient is computed, hence this result follows as the commutative case, which can be seen in \cite{vonzurGathen/Gerhard:2003}.
\end{proof}

\begin{theorem}\label{Effboundi}
Let $f \in D[X;\sigma]$ nonzero. The number of basic operations needed to run Algorithm \ref{alg:bound} belongs to $\bigO\big(rn\operatorname{I}(\mu) + r^3n^2 \operatorname{S}(\mu) + r^2n^2(\operatorname{M}(\mu)+\operatorname{A}(\mu))\big)$, where $n = \deg f$. 
\end{theorem}

\begin{proof}
For each iteration $1 \leq i \leq r$, let $m_i = \deg(c_i)$, $d_i = \deg \rgcd{f}{c_i}$ and $e_i = \deg \rgcd{f^*}{f_{c_i}}$. Since $\deg(f_{c_i}) = n-d_i$ and $\deg(f^*) = in - \sum_{1 \leq j < i} (d_j + e_j)$, it follows that the number of basic operations needed to run Algorithm \ref{alg:bound} is bounded by 
\begin{multline}\label{eq:boundAlgb1}
\sum_{i=1}^r B_{\textsc{ann}}(n,m_i,d_i) + B_{\textsc{ann}} (in - \textstyle\sum_{j<i}(d_j + e_j), n-d_i, e_i) \\
+ B_{\textsc{mult}}(in - \textstyle\sum_{j<i}(d_j + e_j) - e_i, n-d_i) \leq \\
\leq \sum_{i=1}^r B_{\textsc{ann}}(n,m_i,0) + B_{\textsc{ann}} (in, n-d_i, 0) + B_{\textsc{mult}}(in, n)
\end{multline}
We assume $m_1, \dots, m_r$ are constant since the generators $c_1, \dots, c_r$ are fixed for each $D$ and $\sigma$. Moreover $d_i \leq m_i$ for any $i=1,\ldots , r$, so we also assume that $d_1,\ldots, d_r$ are constant.
From Lemma \ref{numops} and Proposition \ref{numops:LEEAAnn} we may obtain how many inversions, sums, multiplications and powers of $\sigma$ in $D$ appear in \eqref{eq:boundAlgb1}. Hence the number of inversions is given by 
\begin{displaymath}
\sum_{i=1}^r (m_i + 1) + (n - d_i + 1) \in \bigO(rn).
\end{displaymath}
The number of sums in given by 
\begin{multline*}
\sum_{i=1}^r \Big( \tfrac{1}{2} (n + m_i + 1) (n + m_i + 4) + \tfrac{1}{2} (in + n - d_i + 1) (in + n - d_i + 4) + in^2 \Big) \\
\in \bigO\big((\tfrac{1}{3}r^3 + 2 r^2 + \tfrac{11}{3}r) n^2 \big).
\end{multline*}
The number of multiplications is
\begin{multline*}
\sum_{i=1}^r \Big( (n - m_i + 1)(m_i + 2) + 2m_i(n + 2) + \\ + (in - n + d_i + 1)(n - d_i + 2) + 2(n - d_i) (in + 2) + (in + 1)(n + 1) \Big) \\
\in \bigO\big((\tfrac{3}{2}r^2 + \tfrac{1}{2}r) n^2\big).
\end{multline*}
Finally, the number of powers of the automorphism is
\begin{multline*}
\sum_{i=1}^r \Big( (n - m_i + 1)(m_i + 2) + m_i (n + 2) + (in - n + d_i + 1)(n - d_i + 2) + (n - d_i) (in + 2) + in(n + 1) \Big) \\
\in \bigO(r^2 n^2).
\end{multline*}
\end{proof}

\begin{theorem}\label{Effboundii}
Let $f \in D[X;\sigma]$. The number of basic operations needed to run Algorithm \ref{alg:bound2} belongs to $\bigO\big( \sqrt{r}sn \operatorname{I}(\mu) + r\sqrt{r}n^3 \operatorname{S}(\mu) + srn^2 (\operatorname{M}(\mu) + \operatorname{A}(\mu)) \big)$, where $n = \deg f$. 
\end{theorem}

\begin{proof}
 Since $\delta = 0$ , we can assume $d_1, \dots, d_{s-1} \in D$ and $d_s = X$ as a direct consequence of \cite[Theorem 1.1.22]{Jacobson:1996} or \cite[Theorem 2.8]{Lam/Leroy:1988}. In each iteration we assume $\deg(f^*) = q$ and $\deg\llcm{f^*}{f^*_{d_i}} = p$. Concerning the \textbf{if} condition, the worst situation is the one in which all $\lrem{f^*d_i}{f^*} = 0$ except the last one. The costs of these remainders, including the computation of $f^* d_i$, are
\begin{equation}\label{boundcostif}
(s-1)\big(B_{\textsc{mult}}(q,0) + B_{\textsc{div}}(q,q) \big) + B_{\textsc{div}}(q+1,q). 
\end{equation}
By Lemma \ref{quitoXm}, we do not loose generality if we assume $\rgcd{f^*}{X}= 1$, then it follows that $\deg(f^*_{d_i}) = q$ and its computation costs
\begin{equation}\label{boundcostann}
B_{\textsc{ann}}(q,\deg d_i,0) \leq B_{\textsc{ann}}(q,1,0),
\end{equation}
since $\deg d_i \leq 1$. Since $\deg\llcm{f^*}{f^*_{d_i}} = p$ and $\deg(f^*_{d_i}) = q$, we have that $\deg\rgcd{f^*}{f^*_{d_i}} = 2q-p$. In order to compute $\llcm{f^*}{f^*_{d_i}}$, we have to replace $f^*_{d_i}$ by $\lrem{f^*_{d_i}}{f^*}$, with cost
\begin{equation}\label{boundcostdiv}
B_{\textsc{div}}(q,q),
\end{equation}
moreover $\rgcd{f^*}{f^*_{d_i}} = \rgcd{f^*}{\lrem{f^*_{d_i}}{f^*}}$ implies that the cost needed to compute $\llcm{f^*}{f^*_{d_i}}$ is
\begin{equation}\label{boundcostllcm}
B_{\textsc{ann}}(q,q-1,2q-p) + B_{\textsc{mult}}(p-q-1,q-1),
\end{equation}
since the worst case is $\deg(\lrem{f^*_{d_i}}{f^*}) = q-1$. We sum \eqref{boundcostif}, \eqref{boundcostann}, \eqref{boundcostdiv} and \eqref{boundcostllcm} to get that the cost of each iteration is bounded by  
\begin{multline}\label{boundcostiter}
(s-1)B_{\textsc{mult}}(q,0) + sB_{\textsc{div}}(q,q) + B_{\textsc{div}}(q+1,q) + B_{\textsc{ann}}(q,1,0) \\
+ B_{\textsc{ann}}(q,q-1,2q-p) + B_{\textsc{mult}}(p-q-1,q-1).
\end{multline}
After we evaluate each term in \eqref{boundcostiter}, we conclude that the number of basic operations in $D$ needed to run an iteration is bounded by 
\begin{multline}
E(p,q) = (p-q + s + 4) \operatorname{I}(\mu) + \big((p-q)(3q+2) + (\tfrac{1}{2}q^2 + \tfrac{2s+13}{2} q + (2s+9)) \big) \operatorname{S}(\mu) + \\
+ \big( (p-q)(3q-4) + ((2s+8)q + (3s+5)) \big) \operatorname{M}(\mu) + \big((p-q)(2q+2) + (s+5)q\big) \operatorname{A}(\mu).
\end{multline}
As bigger $p-q$, bigger number of basic operations, but less number of iterations. Assume $p>k>q$. Observe that $E(p,q) = (p-q) \Phi(q) + \Gamma(q)$ where $\Phi(q)$ is a linear increasing  polynomial and $\Gamma(q)$ is quadratic with all its coefficients positive. Then
\begin{displaymath}
E(p,k) + E(k,q) - E(p,q) = (p-k) (\Phi(k) - \Phi(q)) + \Gamma(k) > 0. 
\end{displaymath}
Hence the biggest number of basic operations is obtained when $p = q+1$ in each iteration. Therefore, the number of basic operations needed to run Algorithm \ref{alg:bound2} is bounded by 
\begin{multline*}
\sum_{q=n}^{\sqrt{r}n - 1} E(q+1,q) = \sum_{q=n}^{\sqrt{r}n - 1} \Big( (s+5) \operatorname{I}(\mu) + \big(\tfrac{1}{2}q^2 + \tfrac{2s+19}{2} q + (2s+11)\big) \operatorname{S}(\mu) \\
+ \big( (2s+11)q + (3s+1) \big) \operatorname{M}(\mu) + \big( (s+7)q + 2 \big) \operatorname{A}(\mu) \Big),
\end{multline*}
which belongs to $\bigO\big( \sqrt{r} s n \operatorname{I}(\mu) + \frac{r \sqrt{r} -1}{6} n^3 \operatorname{S}(\mu) + (2s+11)\frac{r-1}{2} n^2 \operatorname{M}(\mu) + (s+7)\frac{r-1}{2} n^2 \operatorname{A}(\mu) \big)$, where $r$ is the rank of $D[X;\sigma]$ over its center and $s$ is the number of generators of $D[X;\sigma]$ as an algebra over its center.
\end{proof}

\begin{remark}
Although there are not previous algorithms for calculating the bound in this generality, we should mention two cases in which some central polynomials related with the bound are computed. Let \(D = \ffield[q]\), for each \(f \in \ffield[q][X;\sigma]\), a minimal central multiple of \(f\), and hence of \(f^*\), is computed in \cite[Lemma 4.2]{Giesbrecht:1998} with cost \(\bigO^\sim\big(n^2 \mu^3 + n^3 \mu^2 + \operatorname{MM}(n\mu)\big)\). In \cite[pp. 432]{Caruso/LeBorgne:2017} it is computed the reduced norm \(\mathcal{N}(f)\), other multiple of a bound, with cost \(\bigO(n \mu^3)\). Observe that, in this particular case, Theorems \ref{Effboundi} and \ref{Effboundii} say that Algorithms \ref{alg:bound} and \ref{alg:bound2} belong to $\bigO\big(n^2 \mu^7\big)$ and \(\bigO\big( n^3 \mu^4 + n^2 \mu^4 \log^2 \mu \log \log \mu \big)\), respectively, since \(r = \mu^2\) and \(s = 2\). So Algorithm \ref{alg:bound} is faster than the one in \cite{Giesbrecht:1998} in degree (the linear system needed there to compute the minimal central polynomial in Example \ref{finitefields} is \(400 \times 400\)), although it is slower that the norm computation in \cite{Caruso/LeBorgne:2017}. The latter is reasonable since, due to its minimal degree, the bound should take more operations to be computed. The reader might ask about a procedure to obtain the bound once the norm is known. However, this requires, at least, a factorization of \(\mathcal{N}(f)\) as a polynomial in the center, and later, for some polynomials, an exhaustive examination of the divisors. In any case the computation of the norm works only in the finite field case, and the properties used for that are too close to finite field to allow a generalization to other division rings. 
\end{remark}

\section{A criterion of irreducibility}\label{irred}

In this section we provide an algorithmic criterion to decide whether a given bounded polynomial $f \in R = D[X; \sigma, \delta]$ is irreducible. This procedure works in the class of Ore extensions satisfying the conditions \ref{condiciones} below. However, before detailing the algorithm, we consider a shortcut that allows us to accelerate the execution time. The following proposition, that generalizes \cite[Theorem 4.3]{Giesbrecht:1998}, is an important consequence of the proof of Theorem \ref{cotabound}. 

\begin{proposition}\label{boundmaxdegree}
 Let $R = D[X; \sigma, \delta]$ be free of finite rank $r$ over its center $C$. Let $f \in R$ such that  $f^* \in C$. If $f^*$ is a prime twosided polynomial and $\deg f^*=\sqrt{r}\deg f$ then $f$ is irreducible and $\rend{R}{R/Rf}$ is commutative. 
\end{proposition}

\begin{proof}
 We claim that  $R/Rf$ must be indecomposable. To see this, consider the inequality (\ref{eq:bound}) of the proof of Theorem \ref{cotabound}. If $\deg f^* = \sqrt{r} \deg f$, then, by \eqref{eq:bound}, $\deg f^*=\deg g^*h^*$, which is only possible if $Rf^*=Rg^*h^*$, and $f^*$ cannot be then prime twosided. Now, \eqref{eq:ind} applies, and we get thus that $l = 1$. Hence, $f = p_1$, which is irreducible.  Finally we deduce from \eqref{eq:simple} that $\rend{R}{R/Rf} \cong C/Cf^*$, a commutative ring. 
\end{proof}

 The idea of proof of Lemma \ref{centrof^*} is that of the first part of the proof of \cite[Theorem 4.3]{Giesbrecht:1998}, where the center of $R/Rf^*$ is computed in the case $R = \ffield{}[X;\sigma]$ for $\ffield{}$ a finite field. 
\begin{lemma}\label{centrof^*}
Let $f \in R = D[X;\sigma]$ a bounded polynomial of positive degree such that $\rgcd{f}{X} = 1$. Then the center of $R/Rf^*$ is $C/Cf^*$, where $C$ is the center of $R$.
\end{lemma}
\begin{proof}
Let $g + R f^*$ be  in the center of  $R/Rf^*$ with $\deg g < \deg f^*$. Since for all $\alpha \in D$, $\alpha g - g \alpha \in R f^*$ and $\deg(\alpha g - g \alpha) < \deg f^*$ we get that  
\begin{displaymath}
\alpha g - g \alpha = 0 \quad \text{for all $\alpha \in D$}.
\end{displaymath}

We also have that $Xg - gX \in R f^*$ and $\deg(Xg - gX) \leq \deg f^*$, hence
\begin{displaymath}
Xg - gX = \beta f^* \quad \text{for some $\beta \in D$.}
\end{displaymath}
By Lemma \ref{quitoXm}, $\rgcd{f^*}{X} = 1$, the constant coefficient of $f^*$ is nonzero, then it follows that $\beta = 0$ and, therefore, 
\begin{displaymath}
Xg - gX = 0.
\end{displaymath}
We deduce that $g \in C$ because it commutes with $X$ and every element in $D$.
\end{proof}

\begin{remark}\label{casosbuenos}
Let $R = D[X;\sigma]$ be as in Remark \ref{sinderivacion}. By Lemma \ref{quitoXm}, if $\rgcd{f}{X} = 1$ then  $f^*$ belongs to the center $K[z]$ of $R$ up to a scalar in $D$. Since $X^m$ can be easily extracted as a right factor, we do not lose generality if we assume $\rgcd{f}{X} = 1$. As a consequence of \cite[Theorem 2.8]{Lam/Leroy:1988} or \cite[Theorem 1.1.22]{Jacobson:1996}, $f^*$ is a prime twosided polynomial in $R$ if and only if $f^*$ is an irreducible polynomial in $K[z]$.   In the differential case $R = D[X;\delta]$, every twosided polynomial belongs, up to a scalar in $D$, to the center $C$ of $R$ (see \cite[Theorem 1.1.32]{Jacobson:1996}). Note that $C$ may also be a commutative polynomial ring in this case. Concretely, by \cite[Theorem 1.1.32]{Jacobson:1996}, if the derivation is inner or the characteristic of $D$ is positive and it satisfies a polynomial equation, then $C = K[z]$ for some subfield $K$ of $D$ and some central nonconstant polynomial $z \in R$.
\end{remark}

\begin{proposition}\label{endconmu}
Let $R = D[X;\sigma]$ be as in Remark \ref{sinderivacion}, and consider $f \in R$ of positive degree such that $\rgcd{f}{X} = 1$ (hence, $f^*$ belongs to the center $K[z]$ of $R$). Then $f^*$ is irreducible in $K[z]$ and $\deg f^*=\sqrt{r}\deg f$ if an only if $f$ is irreducible in $R$ and $\rend{R}{R/Rf}$ is commutative. 
\end{proposition}
\begin{proof}
Observe that $f^*$ belongs to the center by Lemma \ref{quitoXm}. By Proposition \ref{boundmaxdegree}, we only need to show that if $f$ is irreducible and $\rend{R}{R/Rf}$ is commutative, then $\deg f^* = \sqrt{r} \deg f$. By Lemma \ref{centrof^*}, the center of $R/Rf^*$ is $C/Cf^*$. On the other hand, the center of $\rend{R}{R/Rf} = \rend{R/Rf^*}{R/Rf}$ is the center of $R/Rf^*$, since this last algebra is simple. Thus, the commutativity of $\rend{R}{R/Rf}$ already implies that $\rend{R}{R/Rf} \cong C/Cf^*$. By \eqref{mr} and \eqref{eq:simple}, $\deg f^* = \sqrt{r} \deg f$. 
\end{proof}

\begin{example}
Let $f \in R$ be the polynomial in Example \ref{boundF(t)}. Then $\deg f^* = 6 = 3 \deg f$ and the order of $\sigma$ is $3$. Now, $f^*$ viewed in $C(R)$ is $$\widehat{f}=z^{2} + \left(\frac{(a^{3} + a) s + a^{2} + a + 1}{a^{2} s +
a^{2} + a}\right) z + a^{3} s^{2} + 1,$$
where $s=t^3$. Since $\widehat{f}$ is irreducible in $C(R)$ we can deduce from Proposition \ref{boundmaxdegree} that $f$ is irreducible. 
\end{example}

 Corollary \ref{irr_finitefield} below completes \cite[Theorem 4.3]{Giesbrecht:1998}.

 \begin{corollary}\label{irr_finitefield}
 $R = \ffield{}[X;\sigma]$, where $\sigma$ is an automorphism of order $\mu$ of a finite field $\ffield{}$, and $f \in R$ be such that $\rgcd{f}{X} = 1$. Then $f$ is irreducible if and only if $f^* \in  \ffield{}^{\sigma}[z]$ is irreducible and $\deg f^* = \mu \deg f$. 
\end{corollary}
\begin{proof}
If $f$ is irreducible, then $\rend{R}{R/Rf}$ is a finite division ring over the finite field $C/Cf^*$, where $C = \ffield{}^{\sigma}[z]$. By Little Wedderburn Theorem, $\rend{R}{R/Rf}$ is a commutative field. The corollary now follows from Proposition  \ref{endconmu}.
\end{proof}

\begin{corollary}\label{numirredfact}
Let $R = \ffield{}[X;\sigma]$, where $\sigma$ has order $\mu$, $f \in R$ be such that $\rgcd{f}{X} = 1$, and assume $f^* \in  \ffield{}^{\sigma}[z]$ is irreducible. Then $\deg f^* = \frac{\mu \deg f}{t}$, where $t$ is the number of irreducible factors of $f$. 
\end{corollary}

\begin{proof}
Let $f = p_1 \dots p_t$, with $p_1, \dots, p_t$ irreducible. The result follows by Corollary \ref{irr_finitefield} and the isomorphism of left $R/Rf^*$--modules
\begin{displaymath}
\frac{R}{Rf} \cong \frac{R}{Rp_1} \oplus \dots \oplus \frac{R}{Rp_t}.
\end{displaymath}
\end{proof}

\begin{example}\label{complejos}
 Corollary \ref{irr_finitefield} does not hold if \ffield{} is not finite. Let $\sigma: \Cset \rightarrow \Cset$ be the complex conjugation, and $R = \Cset[X;\sigma]$. Let $f = X^2 + 1 \in \mathbb{R}[X^2] = C(R)$. Then $f^* = f$. Moreover $R/Rf \cong \mathbb{H}$, the Hamilton's quaternions algebra, hence $f$ is irreducible in $R$. But $\deg f^* = \deg f = 2 \neq 4 = 2 \deg f$.
\end{example}

Let $f \in \ffield{}[X;\sigma]$, where $\mu$ is the order of $\sigma$. In the second point of \cite[Proposition 2.1.17]{Caruso/LeBorgne:2017} the reduced norm $\mathcal{N}(f)$ is used to provide an irreducibility criterion. This can be obtained from Corollary \ref{irr_finitefield}.

\begin{corollary} \cite[Proposition 2.1.17]{Caruso/LeBorgne:2017}\label{Caruso}
Let $R = \ffield{}[X;\sigma]$, where $\sigma$ has order $\mu$, and $f \in R$ be such that $\rgcd{f}{X} = 1$. Then $f$ is irreducible if and only if $\mathcal{N}(f)$ is irreducible as a polynomial over $C(R) = \ffield^\sigma[X^\mu]$. 
\end{corollary}

\begin{proof}
Since $\mathcal{N}(f)$ is a twosided multiple of $f$, we get that $f^*$ divides $\mathcal{N}(f)$. By construction, $\deg \mathcal{N}(f) = \mu \deg f$. The result follows now from Corollary \ref{irr_finitefield}.
\end{proof}

Unfortunately, Corollary \ref{Caruso} does not hold if $\ffield{}$ is not a finite field. In fact, the reduced norm of the polynomial considered in Example \ref{complejos} is $\mathcal{N}(X^2+1) = (X^2+1)^2$, which is not irreducible although $X^2+1$ is so.

 \begin{punto}\label{condiciones}\textbf{Assumptions.}
In order to cover the examples described in Remark \ref{casosbuenos}, in  the remaining of this section we assume that $R$ satisfies the following conditions:
\begin{enumerate}
\item $R$ is a free module of finite rank $r$ over its center $C$ (hence, every nonzero polynomial $f \in R$ is bounded). 
\item $C = K[z]$, where $K$ is a commutative subfield of the center of $D$ invariant under $\sigma$ and such that $\delta(K) = 0$, and $z \in R$ is a (central)  nonconstant  polynomial.
\item Every nonzero twosided polynomial of $R$ is, up to multiplication by a constant in $D$, of the form $X^m \alpha$, where $\alpha \in K[z]$ and $m \geq 0$. 
\end{enumerate}
Under these conditions, getting a factorization of a twosided polynomial as a product of prime twosided polynomials is reduced to the computation of a complete factorization of a polynomial of $K[z]$ into irreducibles. In particular, $\alpha \in K[z]$ is prime twosided  if and only if $\alpha$ is irreducible as a polynomial in $K[z]$. Thus, we may consider without loss of generality polynomials $f \in R$ such that $f^* \in C = K[z]$. 
\end{punto}

For any $f \in R$ with $f^* \in C = K[z]$ we have the commutative $K$--algebra $K_f : = K[z]/K[z]f^*$, which is a finite field extension of $K$ if $f^*$ is irreducible in $K[z]$. We know that $R/Rf^*$ is then a simple finite-dimensional $K_f$--algebra. Recall that if $f$ is irreducible in $R$, then $f^*$ is irreducible in $K[z]$. 

\begin{proposition}\label{idempotenteirreducible}
Under the conditions in \ref{condiciones}, let $f \in R$  such that  $f^* \in K[z]$ is irreducible. Let $A = R/Rf^*$ and consider an idempotent $e \in A$ such that $Ae = Aa$, where $a = f + Rf^*$.  Then $f$ is irreducible if and only if the (simple) $K_f$--algebra $(1-e)A(1-e)$ is a skew field.
\end{proposition}

\begin{proof}
We know that $f$ is irreducible if and only if $R/Rf$ is simple as a left $A$--module. On the other hand, $\rend{A}{R/Rf} \cong (1-e)A(1-e)$. 
\end{proof}

We thus deduce from Proposition \ref{idempotenteirreducible} that, if Proposition \ref{boundmaxdegree} fails to check that a given $f \in R$ is irreducible, then we should check if the finite-dimensional simple $K_f$--algebra $(1-e)A(1-e)$ is a division ring. Even thought this is in general a hard problem from the computational point of view \cite{Ronyai:1987}, we will close this section by showing that at least the structure constants of $(1-e)A(1-e)$ can be computed in polynomial time from $f$. We need the following easy lemma.

\begin{lemma}
 Let $A$ be a ring, and $a \in A$. The idempotents $e \in A$ such that $Ae = Aa$ are of the form $e = ya$, where $y \in A$ is a solution of the  equation 
\begin{equation}\label{VN}
a = aya.
\end{equation}
\end{lemma}

\begin{proof}
 Assume $Aa = Ae$ for some idempotent $e \in Aa$. Therefore, $e = ya$ and $a = ze$ for some $y, z \in A$. On the other hand, $a= ae + a(1-e) = ae + ze(1-e) = ae $. Then $y$ is a solution of the equation $aya = a$. Conversely, let $y \in A$ be a solution of the equation $aya = a$. Then $e = ya$ is idempotent and $e \in Aa$. One easily deduces that $Aa = Ae$.
\end{proof}

 \begin{proposition}
Let $R$ satisfy the conditions \ref{condiciones}. Given $f \in R$ with $f^* \in K[z]$ irreducible, there is an algorithm that computes the structural constants of the simple finite-dimensional $K_f$--algebra $(1-e)A(1-e)$ from $f$ in polynomial time in $\deg f$. 
\end{proposition}
\begin{proof}

Let $\{u_1, \dots, u_r\} \subseteq R$ be a basis of $R$ as a module over its center $C = K[z]$. We may calculate the $r^3$ structure constants $c_{ij}^k \in K[z]$ such that
\begin{equation}\label{constantesdeestructura}
u_iu_j = \sum_{k=1}^r c_{ij}^k u_k, \qquad (i,j = 1, \dots, r)
\end{equation}
The commutative polynomials $c_{ij}^k$ in $K[z]$, as well as the expressions \eqref{constantesdeestructura}, may be computed as part of the structure of the ring $R$, and they do not depend on $f$. The bound $f^*$ is computed by means of  Algorithm \ref{alg:bound} or  Algorithm \ref{alg:bound2}. If $f^* \in K[z]$, then we have the field extension $K \subseteq K_f = C/C f^*$. By Lemma \ref{freeness2factor}, $A = R/Rf^*$ is a $K_f$--algebra with basis $\{ \overline{u}_1, \dots, \overline{u}_r \}$, where $\overline{u_i}$ stands for the class of $u_i$ modulo $Rf^*$.
We get from \eqref{constantesdeestructura} that the product of the $K_f$--algebra $A$ is determined by the expressions
\begin{equation}\label{constantesdeestructurahat}
\overline{u}_i\overline{u}_j = \sum_{k=1}^r \widehat{c}_{ij}^k \overline{u}_k, \qquad (i,j = 1, \dots, r),
\end{equation}
where $\widehat{c}_{ij}^k$ denotes the class of $c_{ij}^k$ modulo $C f^*$ in $K_f$. We should be able to do linear algebra computations over $K_f$. In particular, we are assumed to know how to solve systems of linear equations over $K_f$. For instance, once $a = f + Rf^* \in A$ is expressed in coordinates with respect to the basis $\{ \overline{u}_1, \dots, \overline{u}_r \}$, say $(a_1, \dots, a_{r})$,  equation \eqref{VN} leads to a system of linear equations over $K_f$

\begin{equation}\label{VNlinear}
\sum_{j=1}^r \left( \sum_{i,k,l = 1}^r a_i a_l \widehat{c}_{ij}^k \widehat{c}_{kl}^s \right) y_j = a_s \quad \text{for all $1 \leq s \leq r$.}
\end{equation}

We can compute a solution $y$, expressed by its coordinates $(y_1, \dots, y_r)$ in the basis $\{ \overline{u}_1, \dots, \overline{u}_r \}$. From this, we get the idempotent $e = ya$ and, henceforth, $1-e$. 

A system of generators of $(1-e)A(1-e)$ as a vector space over $K_f$ is $\{ (1-e)\overline{u}_1(1-e), \dots, (1-e)\overline{u}_r(1-e) \}$. Hence, we can compute a basis of $(1-e)A(1-e)$ from this set of generators (e.g. selecting a linearly independent subset) and then compute the structure constants corresponding to such a basis. As we have shown, all of this is linear algebra over $K_f$, so it can be done in polynomial time.
\end{proof}

\section{A note about factorization}\label{factoriz}

Let $f \in R = D[X;\sigma,\delta]$  be a bounded polynomial of positive degree. Our aim is to show that, once a bound $f^*$ is computed, we can use a factorization of $f^*$ into prime twosided polynomials (see \cite[Theorem 9, p. 38]{Jacobson:1943}) to obtain a \emph{rough factorization} of $f$, namely, $f = g_1 \cdots g_r$ such that $g_i^*$ is prime twosided for all $i = 1, \dots, r$. In this way, the problem of factorizing $f$ into irreducibles is reduced to the problem of computing zero divisors in the finite-dimensional simple $C/C{g_i}^*$--algebras $R/R{g_i}^*$. 

\begin{proposition}\label{factordenormalfactor}
Let $f \in R = D[X;\sigma,\delta]$ be a bounded polynomial with bound $f^*$, and let $\pi \in R$ be a proper twosided  divisor of $f^*$. Then $p = \rgcd{f}{\pi}$ is a right divisor of $f$ such that $p^* = \pi$. In particular, $f$ has a proper factorization $f = gp$, for some bounded polynomial $g \in R$.
\end{proposition}

\begin{proof}
From the exact sequence of left $R$--modules 
\[
\xymatrix{0 \ar[r] & \displaystyle\frac{Rp}{Rf} \ar[r] & \displaystyle\frac{R}{Rf} \ar[r] & \displaystyle\frac{R}{Rp} \ar[r] & 0},
\]
we get 
\begin{equation}\label{anuladores}
Ann_R\left ( \frac{Rp}{Rf} \right )\cdot Ann_R \left (\frac{R}{Rp} \right ) \subseteq Ann_R \left (\frac{R}{Rf} \right ).
\end{equation}
By hypothesis, there is a proper factorization  $f^* =   \alpha \pi$, where $\alpha$ is a twosided monic polynomial $\alpha\in R$. Since $Rp = Rf + R\pi$,  $\alpha R \pi \subseteq Rf$, and $\alpha R f \subseteq Rf$, we get that $\alpha \in Ann_R(Rp/Rf)$. 
Thus, by \eqref{anuladores}, it yields 
$\alpha p^* \in Rf^*$. Since $f^* = \alpha \pi$, we get that $Rp^* \subseteq R \pi$. But $R \pi \subseteq R p$, which implies that $R \pi \subseteq R p^*$.  Therefore, $p^* = \pi$.  
Finally, since $p$ is a right divisor of $f$, $f = gp$, for some $g \in R$. By the isomorphism of left $R$--modules, 
\[
\frac{Rp}{Rf} = \frac{Rp}{Rgp} \cong \frac{R}{Rg},
\]
$\alpha$ annihilates $R/Rg$, whence $g$ is bounded. 
\end{proof}

\begin{proposition}\label{roughdecomposition}
Let $f \in R = D[X;\sigma,\delta]$ be a bounded polynomial. If we compute a factorization $f^* = \pi_1 \dots \pi_s$ where \(\pi_i\) is prime twosided for all \(1 \leq i \leq s\), then we can compute \(g_1, \dots, g_s \in R\) such that \(f = g_1 \dots g_s\) and \(g_i^* = \pi_i\) for all \(1 \leq i \leq s\). 
\end{proposition}

\begin{proof}
By Proposition \ref{factordenormalfactor} $f = g_1 p_1$, where $p_1^* = \pi_2 \dots \pi_s$ and $g_1$ is bounded with bound dividing $\pi_1$.  Since $\pi_1$ is prime it follows that $g_1^* = \pi_1$. We repeat the process with $p_1$ and its bound $\pi_2 \dots \pi_s$. So, we finally obtain $f = g_1 \dots g_s$ where $g_i^* = \pi_i$.
\end{proof}

\begin{remark}
Twosided polynomials are also known as invariant polynomials, see e.g. \cite[Definition 2.1]{Lam/Leroy:1988}, where semi-invariant polynomials are also defined. A decomposition of an invariant polynomials as a non trivial product of semi-invariant polynomials, would eventually help to get complete factorizations, following the ideas of Proposition \ref{roughdecomposition}.
\end{remark}

Observe that we only have to compute one bound, $f^*$. Then the complete factorization of $f$ into irreducibles is obtained by factorizing each $g_i$. Therefore, the following question arises: \emph{it is possible to compute a factorization into irreducibles of a bounded polynomial $g \in R$ such that $g^*$ is prime?} Next proposition complements Proposition \ref{idempotenteirreducible}.

\begin{proposition}\label{simpledecomposition}
Under the conditions in \ref{condiciones}, let $f \in R$ such that $f^* \in K[z]$ is irreducible. Let $A = R/Rf^*$, $a = f + R f^* \in A$, and $e \in A$ is an idempotent such that $Ae = Aa$. Then $f$ is reducible if and only if the $K_f$--algebra $(1-e)A(1-e)$ has zero divisors. Moreover, each zero divisor allows to compute a proper factorization \(f = h_1 h\) where \(f^* = h^* = h_1^*\).
\end{proposition}

\begin{proof}
  By Proposition \ref{idempotenteirreducible}, \(f\) is reducible if and only if \((1-e)A(1-e)\) is not a division ring, that is, $(1-e)A(1-e)$ has zero divisors. 
Given a zero divisor $\zeta \in (1-e)A(1-e)$,  a proper factorization of $f$ is computed as follows. Write $\zeta = g + Rf^*$, for some polynomial $g \in R$. We have the following chain of strict inclusions of left ideals of $R$,
 \begin{equation}\label{estrictas}
Rf \subsetneq Rf + Rg \subsetneq R. 
\end{equation}
The inclusions are strict since, if we factor out the chain by the ideal $Rf^*$, we get the proper chain of left ideals of the simple algebra $A$ given by  $Ae \subsetneq Ae \oplus A\zeta \subsetneq  A$. Now, $Rf + Rg = Rh$ for $h = \rgcd{f}{g}$, so $h$ is a proper right factor of $f$. In this way, we obtain a proper factorization $f = h_1 h$.  Furthermore, since $Rf^* \leq Rh$ and $f^*R \leq h_1R$, $h^* = f^*=h_1^*$.
\end{proof}

We may repeat the above procedure in the $C/Cf^*$--algebra $R/Rf^*$ by replacing $f$ by its factors $h$ and $h_1$. An obvious recursive process will lead to a complete factorization of $f$ into irreducible factors, as we show in Algorithm \ref{generalfactorization}. In conclusion, we will be able to obtain a complete factorization of $f$ whenever an algorithm to find zero divisors of finite-dimensional simple algebras is available. So, assume that an algorithm $\operatorname{FindZeroDivisor}$ is available, which computes a zero divisor, if it exists, of an input algebra of dimension $m$ over some field with cost $\chi(m)$. For instance, in skew polynomials over finite fields, there exist Las Vegas algorithms developed by Giesbrecht \cite{Giesbrecht:1998}, R\'onyai \cite{Ronyai:1987} or \cite{Caruso/LeBorgne:2017}. Nevertheless, as pointed out above, in general, such a problem is very hard from a computational point of view. 

\begin{remark}
We know that proper factorizations of $f$ correspond to $R/Rf^*$--submodules of $R/Rf$, and a complete factorization of $f$ into irreducibles is nothing but a composition series of $R/Rf$. In this way, the connection of $R/Rf^*$ with the factorization theory of $f$ is much closer than that of the eigenring $\mathcal{E}(Rf)$. Moreover, in general, there is no ring homomorphisms between $R$ and $\mathcal{E}(Rf)$, in contrast with the tight relationship between $R$ and $R/Rf^*$. For instance, the structure constants for the multiplication in $R/Rf^*$ come directly from the multiplication of $R$. Thus, our methods differ from that of \cite{Giesbrecht:1998} and \cite{Giesbrecht/Zhang:2003}, based on an iterative use of the eigenring. We thus obtain different algorithms. This is clear for Algorithm \ref{generalfactorization}. Let us compare Algorithm \ref{factboundirred} with the eigenring method. Since the eigenring of $R/Rf$ is isomorphic to $\rend{A}{R/Rf}$, we get that, if $f^*$ is irreducible, then $\mathcal{E}(Rf) \cong (1-e)A(1-e)$. A first factor $h$ of $f$ can be then computed either by finding a zero divisor of $\mathcal{E}(Rf)$, abstracting the ideas from \cite{Giesbrecht:1998} or \cite{Giesbrecht/Zhang:2003}, or of $(1-e)A(1-e)$, as in Proposition \ref{simpledecomposition}. But if the factor $h$ has to be further factorized, we have two different alternatives. The eigenring method will require to compute $\mathcal{E}(Rh)$, and it is not clear that the previous computation of $\mathcal{E}(Rf)$ will help to this task, and then find a zero divisor of $\mathcal{E}(Rh)$. Our method will require to replace $f$ by $h$ and, since $h^* = f^*$, apply Proposition \ref{idempotenteirreducible} by using the same ring $A = R/Rf^*$ which is already computed. Thus, both Algorithm \ref{generalfactorization} and Algorithm \ref{factboundirred} are different in nature from the algorithms of factorization from  \cite{Giesbrecht:1998} or \cite{Giesbrecht/Zhang:2003}, and they will give genuine alternatives to the latter when applied to the examples of Ore extension considered there (see also Section \ref{otherfact}).
\end{remark}

\begin{algorithm}
\caption{Factorize}\label{generalfactorization}
\begin{algorithmic}
\REQUIRE $f \in R = D[X;\sigma,\delta]$ 
\textbf{Assumptions:} $R$ satisfies conditions \ref{condiciones}.  $f^* \in K[z]$. An algorithm for factorizing polynomials in $K[z]$ is available. An algorithm for factorizing $g$ with $g^* \in K[z]$ irreducible is available.
\ENSURE A list of irreducible polynomials $[p_1,\ldots, p_l]$ such that $f=p_1\cdots p_l$.
\STATE{$output \gets []$} \COMMENT{the empty list}
\STATE{Compute a bound $f^*$}
\STATE{Factorize $f^*=\pi_1\cdots \pi_s$}
\STATE{$i \gets 1$, $f_1 \gets f$}
\WHILE{$i < s$}
	\STATE{$f_{i+1} \gets \rgcd{f_i}{\pi_{i+1} \dots \pi_s}$}
	\STATE{$g_i \gets \lquot{f_i}{f_{i+1}}$}
	\STATE{$i \gets i + 1$}
\ENDWHILE
\STATE{$g_s \gets f_s$}
\FOR{$i \gets 1$ \TO $s$}
	\STATE{$output \gets output + \operatorname{FactorizeIrred}(g_i,\pi_i)$} \COMMENT{where $+$ denotes concatenation of lists}
\ENDFOR
\RETURN $output$
\end{algorithmic}
\end{algorithm}

\begin{algorithm}
\caption{FactorizeIrred}\label{factboundirred}
\begin{algorithmic}
\REQUIRE $f \in R = D[X;\sigma,\delta]$ with $f^* \in K[z]$ irreducible. 
\textbf{Assumptions:} $R$ satisfies conditions \ref{condiciones}.  There is an algorithm $\operatorname{FindZeroDivisor}$ which returns a zero divisor of its input, if it exists, or $0$ otherwise.
\ENSURE A list of irreducible polynomials $[p_1,\ldots, p_l]$ such that $f=p_1\cdots p_l$.
\IF{$f = 1$} 
	\STATE{$output = []$}
\ELSIF{$\deg(f^*) = \sqrt{r} \deg(f)$}
	\STATE{$output = [f]$}
\ELSE
	\STATE{Compute a solution $y$ in $A = R/R f^*$ of the equation $f = fyf \pmod{Rf^*}$.}
	\STATE{$e \gets y f$}
	\STATE{$aux \gets \operatorname{FindZeroDivisor}((1-e)A(1-e))$}
	\STATE{$g \gets \rgcd{aux}{f}$}
	\STATE{$h \gets \lquot{f}{g}$}
	\STATE{$output \gets \operatorname{FactorizeIrred}(h,f^*) + \operatorname{FactorizeIrred}(g,f^*)$} 
\ENDIF
\RETURN $output$
\end{algorithmic}
\end{algorithm}

As a consequence of the previous discussions and Propositions \ref{roughdecomposition} and \ref{simpledecomposition} we have proved the following

\begin{theorem}
Let $R = D[X;\sigma,\delta]$ that satisfies conditions \ref{condiciones}. Let $f \in R$ with $f^* \in K[z]$. Then Algorithm \ref{generalfactorization} correctly factorizes $f$ into a product of irreducibles in $R$. 
\end{theorem}

\begin{example}\label{ex:factF(t)}
Let $R=D[X;\sigma]$ as described in Example \ref{boundF(t)}, i.e., $D=\ffield[16](t)$ is the field of rational functions over $\ffield[16]$ and $\sigma:D \rightarrow D$ is defined by $\sigma(t) = a^5t$. We recall the reader that the  order of $\sigma$ is $3$, the invariant subfield is $D^\sigma = \ffield[16](t^3)$ and $C= C(R) = D^\sigma[X^3]$. Hence, the rank of $R$ over $C$ is nine and, in particular, by Theorem \ref{cotabound}, $\deg g^*\leq 3\deg g$ for any $g\in R$. For brevity, we shall write the elements of $\ffield[16]$, different of 0 and 1, as powers of the primitive element $a$, which verifies $a^4=a+1$, and not as polynomials over $\ffield[2]$. Let $f\in R$ be the polynomial 
$$f=X^{3} + \left(\frac{a^{11} t^{2} + a^{12} t + 1}{t + a}\right) X^{2} + \left(\frac{a t^{3} +
a^{2} t^{2} + a^{13} t + a}{t + a}\right) X +
a^{2} t^{3} + a t.$$
By Algorithm \ref{alg:bound} or \ref{alg:bound2}, we may calculate its bound $f^*$ and get that
$$f^*=X^{9} + \left(\frac{a^8 t^{6} + t^{3} + a^{13}}{a^5 t^{3} + a^8}\right) X^{6} +
\left(\frac{t^{9} + a^4 t^{6} + a^9
t^{3} + 1}{a^{12} t^{3} + 1}\right) X^{3} +
a^6 t^{9} + a^{3} t^{3},$$
or, viewed under the change of variables $s=t^3$ and $z=X^3$,
$$\widehat{f}=z^{3} + \left(\frac{a^8 s^{2} + s + a^{13}}{a^5 s + a^8}\right) z^{2} +
\left(\frac{s^{3} + a^4 s^{2} + a^9 s
+ 1}{a^{12} s + 1}\right) z + a^6 s^{3} + a^{3} s.$$
Now, observe that $\widehat{f}$, as an element in $C$, factors out into irreducibles  as $\widehat{f}=\widehat{p_1}\widehat{p_2}$, where
$$\widehat{p_1}=z + a^{3} s$$
and 
$$\widehat{p_2}=z^{2} + \left(\frac{a^4 s + a^5}{a^{12} s + 1}\right) z + a^{3} s^{2} + 1.$$
Hence, by Proposition \ref{factordenormalfactor}, the factorization of $\widehat{f}$ give us a proper factorization $f=g_1g_2$ in $R$, where $g_2=(p_2,f)_r$ and $g_1=\lquot{f}{g_2}$. Here, $p_1$ and $p_2$ denote the polynomials $\widehat{p_1}$ and $\widehat{p_2}$, respectively, after undoing the above change of variables. Concretely,
$$g_1=X + \frac{a^{2} t^{8} + a^{2} t^{6} + a t^{5} + a^8
t^{4} + a^8 t^{2} + a^{13}
t}{a^6 t^{7} + a^{11}
t^{5} + a^5 t^{4} + a^7 t^{3}
+ a^{12} t + a^{12}}$$
and 
\[
\begin{split}
g_2&=X^{2} + \left(\frac{a^4 t^{6} +a^{12} t^{5} + t^{4} + a^4 t^{3} + a^{3} t^{2} + a^{10} t + a^{11}}{a t^{7} + a t^{5} + t^{4} + a^7 t^{3} + a^7 t + a^{12}}\right) X \\
&\quad + \frac{a^{12} t^{9} +a^9 t^{7} + a^{11} t^{6} + a^{12} t^{5} + a^{10} t^{4} +a^{10} t^{3} + a^{3} t^{2} +
a^{2} t + a^{2}}{a^6 t^{7} + a^6 t^{5} + a^5 t^{4} + a^{12} t^{3} + a^{12} t + a^{2}}.
\end{split}
\]
Furthermore, $p_1$ and $p_2$ are bounds of $g_1$ and $g_2$, respectively. Clearly, $g_1$ is irreducible. Now, since $p_2$ is irreducible, as polynomial in the center, and $\deg p_2=3\deg g_2$, by Proposition \ref{boundmaxdegree}, $g_2$ is also irreducible. Hence, the factorization of $f$ in $R$ is completed.
\end{example}

\begin{example}\label{ex:factquat}
Let us consider the skew polynomial algebra $R=D[X;\sigma]$ described in Example \ref{ex:quaternions}, i.e., $D=\mathbb{Q}\oplus \mathbb{Q}i \oplus \mathbb{Q} j \oplus \mathbb{Q} k$, where $i^2=-1$, $j^2=-1$ and $ij=-ji=k$, and $\sigma$ is the inner automorphism given by $\sigma(a)=u a u^{-1}$, where $u=1+i$. As observed in Example \ref{ex:quaternions}, $\sigma$ has order one with respect to an inner automorphism, and its invariant subfield is $D^\sigma=\mathbb{Q}\oplus \mathbb{Q}i$, the center of $R$ is $C = C(R)=\mathbb{Q}[z]$, where $z=u^{-1}X$, and a basis of $R$ over $C$ is given by $\{1,i,j,k\}$. Let us factorize the polynomial
\begin{displaymath}
\begin{split}
f &= \left(3 - 4i + \tfrac{1}{2}j + \tfrac{7}{2}k\right)X^7 + \left(-\tfrac{77}{6} - \tfrac{313}{42}i + \tfrac{80}{21}j - \tfrac{37}{14}k\right)X^6 \\
&\quad + \left(-\tfrac{211}{42} + \tfrac{919}{42}i - \tfrac{110}{21}j + \tfrac{61}{21}k\right)X^5 + \left(\tfrac{215}{7} - \tfrac{23}{21}i + \tfrac{107}{7}j - \tfrac{575}{42}k\right)X^4 \\
&\quad + \left(\tfrac{22}{21} - \tfrac{149}{7}i - \tfrac{283}{21}j + \tfrac{347}{21}k\right)X^3 + \left(-\tfrac{36}{7} - \tfrac{13}{21}i - \tfrac{22}{3}j - \tfrac{129}{7}k\right)X^2 \\
&\quad + \left(-\tfrac{62}{21} + \tfrac{58}{7}i + \tfrac{90}{7}j - \tfrac{58}{21}k\right)X - \tfrac{40}{21} + \tfrac{80}{21}i - \tfrac{40}{7}j + \tfrac{80}{21}k
\end{split}
\end{displaymath}
using Algorithm \ref{generalfactorization}. Algorithm \ref{alg:bound} computes a bound $f^*$, and after making the change of variable $z=u^{-1}X$ we obtain
\begin{displaymath}
\begin{split}
\widehat{f} &= 4096 z^{12} - \tfrac{684032}{175}z^{11} + \tfrac{2526208}{315}z^{10} - \tfrac{8259584}{1575}z^9 + \tfrac{830464}{105}z^8 - \tfrac{1119232}{225}z^7 \\
&\quad + \tfrac{8461568}{1575}z^6 - \tfrac{905728}{315}z^5 + \tfrac{226816}{105}z^4 - \tfrac{436736}{525}z^3 + \tfrac{160256}{225}z^2 - \tfrac{111616}{1575}z + \tfrac{2048}{35} \\
&= 4096 \left(z^2 - \tfrac{22}{21}z + \tfrac{20}{21}\right) \left(z^4 + \tfrac{1}{3}z^3 + \tfrac{13}{12}z^2 + \tfrac{1}{3}z + \tfrac{3}{8}\right) \\
&\quad \cdot \left(z^6 - \tfrac{6}{25}z^5 + \tfrac{1}{10}z^4 - \tfrac{8}{25}z^3 + \tfrac{9}{25}z^2 - \tfrac{1}{25}z + \tfrac{1}{25}\right),
\end{split}
\end{displaymath}
which yields the following factorization of $f = g_1 g_2 g_3$, where
\begin{align*}
g_1 &= (3 - 4i + \tfrac{1}{2}j + \tfrac{7}{2}k)X^2 + (-\tfrac{22}{3} + \tfrac{22}{21}i - \tfrac{88}{21}j - \tfrac{22}{7}k)X + \tfrac{160}{21} + \tfrac{40}{7}i + \tfrac{20}{3}j - \tfrac{20}{21}k \\
g_2 &= X^2 + (\tfrac{3}{10} + \tfrac{1}{30}i - \tfrac{14}{15}j - \tfrac{11}{15}k)X - \tfrac{16}{15} + \tfrac{1}{3}i - \tfrac{1}{2}j + \tfrac{1}{30}k \\
g_3 &= X^3 + (\tfrac{8}{25} - \tfrac{14}{25}i + \tfrac{13}{25}j + \tfrac{14}{25}k)X^2 + (-\tfrac{6}{25} - \tfrac{2}{5}i - \tfrac{2}{5}j - \tfrac{8}{25}k)X + \tfrac{6}{25} + \tfrac{4}{25}i + \tfrac{2}{25}j - \tfrac{12}{25}k.
\end{align*}
We have to apply Algorithm \ref{factboundirred} to $g_1,g_2,g_3$ and their corresponding bounds. For $i = 2,3$, $\deg g_i^* = 2 \deg g_i$, hence we conclude that $g_2$ and $g_3$ are irreducible. It remains to factor out $g_1 = g_1^*$. Let $K_{g_1} = \mathbb{Q}[z]/\langle \widehat{g_1} \rangle$, where $\widehat{g_1} = z^2 - \tfrac{22}{21} z + \frac{20}{21}$. Then $A = R/Rg_1^*$ is isomorphic to the standard quaternion algebra over $K_{g_1}$. So the problem is reduced to find a zero divisor, if it exists, in $A$. By \cite[Chapter V, \S 57]{OMeara:2000} this problem is equivalent to find elements $a,b \in K_{g_1}$ such that $a^2 + b^2 = -1$. This problem is expected to be hard, since Ronyai proves in \cite{Ronyai:1988} that the computation of zero divisors in quaternion algebras over $\mathbb{Q}$ is $\text{NP} \cap \text{co-NP}$. However, in this example, we can find zero divisors if we can compute a rational point in the variety over $\mathbb{Q}$ defined by the ideal
\begin{displaymath}
\left\langle \tfrac{22}{21}a_1^2 + 2 a_1 a_2 + \tfrac{22}{21}b_1^2 + 2 b_1 b_2, a_2^2 - \tfrac{20}{21}a_1^2 + b_2^2 - \tfrac{20}{21}b_1^2 +1 \right\rangle \subseteq \mathbb{Q}[a_1,a_2,b_1,b_2].
\end{displaymath}
A clever use of ``brute force" and Groebner bases techniques yields to 
$\left(\tfrac{63}{10}, -5, -\tfrac{21}{10},-4\right)$, which corresponds to the zero divisor $\left(\tfrac{63}{10}t - 5\right) - \left(\tfrac{21}{10}t +4\right)i + j$. Recall that $z = u^{-1}X$, hence 
\begin{displaymath}
h = \left(\tfrac{63}{10}u^{-1}X - 5\right) - \left(\tfrac{21}{10}u^{-1}X +4\right)i + j = \left(\tfrac{21}{10} - \tfrac{21}{5}i\right)X - 5 - 4i + j
\end{displaymath}
is a right factor of $g_1$, since 
\begin{displaymath}
g_1 = \left(\left(\tfrac{22}{21} + \tfrac{4}{21}i + \tfrac{5}{7}j + \tfrac{5}{21}k\right)X - \tfrac{190}{147} + \tfrac{10}{441}i - \tfrac{470}{441}j - \tfrac{290}{441}k \right) \left(\left(\tfrac{21}{10} - \tfrac{21}{5}i\right)X - 5 - 4i + j\right),
\end{displaymath}
and the factorization is completed.
\end{example}

\begin{example}
Let $R=\mathbb{F}_2(t)[X;\delta]$, where $\delta$ is the usual derivative. We recall that the center of $R$ is $C(R)=\mathbb{F}_2(s)[z]$, where $s=t^2$ and $z=X^2$. Let now 
$$f=X^{2} + \left(\frac{1}{t^{2} + t}\right) X + \frac{1}{t^{3} + t}.
$$ 
We follow the steps of Algorithm \ref{generalfactorization}. We firstly need to compute the bound of $f$. By applying Algorithm \ref{alg:bound}, a bound $f^*=X^4$, so $\widehat{f}=z^2=z\cdot z$. Then $f$ is reducible, and we may find a factor as $g_1=\rgcd{f}{X^2}=X + \frac{1}{t + 1}$. So $f=g_2g_1$, where $g_2=X + \frac{1}{t}$. Obviously, $g_1$ and $g_2$ are irreducible.
\end{example}

\section{Other approaches to factorization}\label{otherfact}

\subsection{Finite fields} \label{ffield}

The factorization of Ore polynomials over a finite field $\mathbb{F}$ is developed in \cite{Giesbrecht:1998} and \cite{Caruso/LeBorgne:2017}. Both references treat the pure automorphism case because,  as $\mathbb{F}$ has no nonzero derivations, by a change of variable, we may reduce the problem to a ring \(R = \mathbb{F}[X;\sigma]\), where $\sigma$ is an automorphism of $\ffield$ of order $\mu$.
Given $f  \in R$ with $\rgcd{f}{X} = 1$, we remind the reader that Algorithm \ref{generalfactorization} consists of two steps: first we compute a rough decomposition $f = g_1 \cdots g_r$ from the complete factorization of the bound $f^* = \pi_1 \cdots \pi_r$ over $\mathbb{F}^{\sigma}[X^{\mu}]$, and then we factorize each $g_i$ with $g_i^* =\pi_i$ by finding zero-divisors in some simple finite algebras over the splitting field of $\pi_i \in \mathbb{F}^{\sigma}[X^{\mu}]$. For this second step, we may use any of the algorithms in \cite{Ronyai:1987,Giesbrecht:1998,Caruso/LeBorgne:2017}.

In \cite{Giesbrecht:1998} the technique of factorization is based on the results of Ronyai \cite{Ronyai:1987} for finding zero divisors in any finite-dimensional algebra over a finite field. Concretely, for any polynomial $f\in R$ of degree $n\in \mathbb{N}$, the so-called eigenring $\mathcal{E}(Rf)$ \cite[\S 0.4]{Cohn:1971} is constructed. Recall that this eigenring is isomorphic to the endomorphism ring of the left $R$--module $R/Rf$. 
Then, a non-trivial zero divisor in $\mathcal{E}(Rf)$, as an algebra over the subfield of invariants $K=\ffield^\sigma$, provides a non-trivial decomposition of $f = gh$. The process is then applied to each of the factors $g, h$, including the computation of the structure constants of each of the eigenrings $\mathcal{E}(Rg)$ and $\mathcal{E}(Rh)$. The run-time needed for computing a complete factorization of $f$ belongs to $\bigO^\sim(n^4)$. This method is translatable to other contexts whenever the invariant subfield, the eigenring and the algorithm for finding zero divisors in it are computable, see next subsection below.

 We shall understand that Berlekamp's algorithm is used to factorize commutative polynomials over finite fields. Actually, we also assume the use of the basic operational algorithms whose complexity is calculated in Section \ref{complexity} in a more general setting than the finite field case.  Although there are faster algorithms for computing these calculations (see, for instance, \cite{vonzurGathen/Panario:2001}, for faster factorization algorithms, or Karatsuba's product), our aim here is to compare our method with the one in \cite{Giesbrecht:1998}, and show an alternative factorization algorithm which reduces the complexity with respect to the degree of the polynomial. 

\begin{lemma}\label{findzerodivisor}
Let $R = \ffield{}[X;\sigma]$, $\mu$ be the order of $\sigma$ and $f \in R$ with $\deg(f) = n$. Algorithm \emph{\texttt{FindZeroDivisor}} in \cite[\S 5]{Giesbrecht:1998} requires 
\[
\bigO(n^2 \mu^2 (\operatorname{S}(\mu)+\operatorname{M}(\mu)+\operatorname{A}(\mu)) + \operatorname{MM}(t^2) (\operatorname{S}(\tfrac{n}{t})+\operatorname{M}(\tfrac{n}{t})) + \operatorname{M}(t^2) \log(t^2) (\operatorname{S}(\tfrac{n}{t})+\operatorname{M}(\tfrac{n}{t})) \tfrac{n}{t})
\] 
operations in $K = \ffield{}^\sigma$ to compute a zero divisor of $(1-e)A(1-e)$ as an algebra over $K_f = C/Cf^*$, where $A = R/Rf^*$, $e \in A$ is an idempotent such that  $Ae = A(f+ Rf^*)$  and $t$ is the number of irreducible factors of $f$. It returns ``Failure" with probability at most $\frac{8}{9}$.
\end{lemma}

\begin{proof}
It follows directly from \cite[Theorem 5.1]{Giesbrecht:1998}, since the dimension of $(1-e)A(1-e) \cong \mathcal{M}_t(K_f)$ is $t^2$, the cost with respect to $K=\ffield{}^\sigma$ of the basic operations in $(1-e)A(1-e)$ is bounded by $n^2 \mu^2 (\operatorname{S}(\mu)+\operatorname{M}(\mu)+\operatorname{A}(\mu))$, and the number of basic operations in $K_f$ with respect to $K$ is written as $S(\frac{n}{t})$ sums and $M(\frac{n}{t})$ multiplications. The ``Failure" probability follows from \cite[Corollary 5.10]{Giesbrecht:1998}.
\end{proof}

\begin{proposition}\label{numops:FactBndIrr}
Let $R = \ffield{}[X;\sigma]$, $\mu$ be the order of $\sigma$, and $f \in R$ such that $\deg(f) = n$, its bound $f^*$ is irreducible and $f$ has $t$ irreducible factors, where $\deg f^* = \frac{n\mu}{t}$. Then Algorithm \ref{factboundirred} requires 
\begin{multline*}
\bigO(n \mu^3 (\operatorname{S}(\mu) + \operatorname{M}(\mu)) + n^2 \mu^2 t (\operatorname{S}(\mu)+\operatorname{M}(\mu)+\operatorname{A}(\mu)) \\
+ t \operatorname{MM}(\mu^2) (\operatorname{S}(\tfrac{n}{t})+ \operatorname{M}(\tfrac{n}{t})) + \operatorname{M}(t^2) \log(t^2) (\operatorname{S}(\tfrac{n}{t}) + \operatorname{M}(\tfrac{n}{t})) n)
\end{multline*}
 operations in $K = \ffield{}^\sigma$ to correctly factorize $f$ into a product of irreducibles in $R$. 
\end{proposition}

\begin{proof}
The cost of the computation of $e$ is bounded by the arithmetic in $R$ and the computation of the solution of a linear system over $K_f$ of size $\mu^2$. Hence, it costs $\bigO(\frac{n^2 \mu^2}{t^2}(\operatorname{S}(\mu)+\operatorname{M}(\mu)+\operatorname{A}(\mu)) + \frac{n\mu^3}{t}(\operatorname{S}(\mu) + \operatorname{M}(\mu)) + \operatorname{MM}(\mu^2)(S(\frac{n}{t})+M(\frac{n}{t})))$ operations in $K=\ffield^\sigma$. By Lemma \ref{findzerodivisor}, the computation of a zero divisor in $(1-e)A(1-e)$ can be achieved in $\bigO(n^2 \mu^2 (\operatorname{S}(\mu)+\operatorname{M}(\mu)+\operatorname{A}(\mu)) + \operatorname{MM}(t^2) (\operatorname{S}(\frac{n}{t}) + \operatorname{M}(\frac{n}{t})) + \operatorname{M}(t^2) \log(t^2) (\operatorname{S}(\frac{n}{t}) + \operatorname{M}(\frac{n}{t})) \frac{n}{t})$ operations in $K$. The cost of the computation of the right greatest common divisor and the left quotient are also in $\bigO(\frac{n^2\mu}{t}(\operatorname{S}(\mu)+\operatorname{M}(\mu)+\operatorname{A}(\mu)))$. It remains to analyze the recursion step. Observe that $f$ have to be split $t-1$ times and the worst case happens when each zero divisor provides an irreducible factor. Hence, we have to call the algorithm $t-1$ times and, in each iteration, the degree of the polynomial under factorization is $\frac{n(t-j)}{t}$ and the number of its  irreducible factors is $t-j$ for $j = 1, \dots, t-1$. Hence, the number of basic operations in $K$ needed to run Algorithm \ref{factboundirred} is in 
\begin{multline*}
\bigO\bigg( \sum_{j=1}^{t-1} \Big[ \tfrac{n \mu^3}{t}(\operatorname{S}(\mu) + \operatorname{M}(\mu)) + n^2 \mu^2 (\operatorname{S}(\mu)+\operatorname{M}(\mu)+\operatorname{A}(\mu)) + \\ 
 + \operatorname{MM}(\mu^2) (\operatorname{S}(\tfrac{n}{t}) + \operatorname{M}(\tfrac{n}{t})) + \operatorname{M}((t-j)^2) \log((t-j)^2) (\operatorname{S}(\tfrac{n}{t}) + \operatorname{M}(\tfrac{n}{t})) \tfrac{n}{t} \Big] \bigg) \subseteq \\
 \subseteq \bigO \bigg( n \mu^3 (\operatorname{S}(\mu) + \operatorname{M}(\mu)) + n^2 \mu^2 t (\operatorname{S}(\mu)+\operatorname{M}(\mu)+\operatorname{A}(\mu)) \\
 + t \operatorname{MM}(\mu^2) (\operatorname{S}(\tfrac{n}{t}) + \operatorname{M}(\tfrac{n}{t})) + \operatorname{M}(t^2) \log(t^2) (\operatorname{S}(\tfrac{n}{t}) + \operatorname{M}(\tfrac{n}{t})) n \bigg).
\end{multline*}
\end{proof}

\begin{theorem}\label{numops:Factorize}
Let $R = \ffield{}[X;\sigma]$, $\mu$ be the order of $\sigma$ and $f \in R$ with $\deg(f) = n$. Then Algorithm \ref{generalfactorization} requires 
\begin{multline*}
\bigO\big( \operatorname{MM}(n) + M(n) \log n \log \# K + n\operatorname{I}(\mu) + \mu^3 n^3 \operatorname{S}(\mu) + \mu^3 n^2(\operatorname{M}(\mu) + \operatorname{A}(\mu)) + \\
 + t \operatorname{MM}(\mu^2) (\operatorname{S}(\tfrac{n}{t})+\operatorname{M}(\tfrac{n}{t})) + \operatorname{M}(t^2) \log t (\operatorname{S}(\tfrac{n}{t}) + \operatorname{M}(\tfrac{n}{t})) n \big)
\end{multline*}
operations in $K = \ffield{}^\sigma$ to correctly factorize $f$ into a product of irreducibles in $R$.
\end{theorem}

\begin{proof}
By Theorem \ref{Effboundii}, $f^*$ can be computed with $\bigO\big(\mu n\operatorname{I}(\mu) + \mu^3 n^3 \operatorname{S}(\mu)+ \mu^2 n^2 (\operatorname{M}(\mu)+\operatorname{A}(\mu))\big)$ operations in $K$. Using Berlekamp's algorithm, $\widehat{f}$, whose degree is at most $n$ as a polynomial over $C = K[X^\mu]$, can be factorized with $\bigO(\operatorname{MM}(n) + \operatorname{M}(n) \log(n) \log \# K)$ operations in $K$. The \textbf{while} loop can be done in the worst case, i.e. when $\widehat{f}$ factorizes as product of linear polynomials, in $\bigO(n\operatorname{I}(\mu) + n^3 \mu^2 (\operatorname{S}(\mu) + \operatorname{M}(\mu) + \operatorname{A}(\mu)))$ basic operations in $K$. The worst case in the \textbf{for} loop holds when $f^*$ is irreducible as a polynomial in $C$, so a bound of the number of operations in $K$ is given in Proposition \ref{numops:FactBndIrr}.
\end{proof}

\begin{corollary}\label{O(fact)}
Let $R = \ffield{}[X;\sigma]$ and $f \in R$ with $\deg(f) = n$. The runtime of  Algorithm \ref{generalfactorization} belongs to $\bigO(n^3)$ operations in $K=\ffield{}^\sigma$.
\end{corollary}

The underlying idea of this improvement in the runtime consists of delaying the use of Ronyai \cite{Ronyai:1987} algorithm for finding zero divisors, or its variant described in \cite{Giesbrecht:1998}. This is so since we find a previous ``rough'' decomposition via the computation of a bound, which can be done faster and it only needs to be computed once. Then, when applying \emph{\texttt{FindZeroDivisor}}, the degree of the polynomials under consideration are always smaller.

Let now us compare our algorithm with the one in \cite{Caruso/LeBorgne:2017}.  The authors assert that this algorithm is as efficient as the commutative factorization whenever the degree of the polynomials is the only variable under consideration. Hence, the implementation of our algorithm in the finite field case cannot improve this theoretical runtime. For this reason, it has no sense to give a detailed study of the efficiency using their techniques for basic operations. Nevertheless, we may analyze the steps given there.  Essentially, they share the first step computing a left multiple in the center of the polynomial $f\in R$. Obviously, the minimality  of a bound yields that $f^*$ needs more operations to be computed than the norm $\mathcal{N}(f)$, which in general is a multiple of $f^*$. In opposition to this, we need to factorize a commutative polynomial of lower degree. Also, any decomposition of a bound produces a computable decomposition of the polynomial, whilst this is not always possible with the norm (the so-called $(e)$ type polynomials). Unfortunately, the method in \cite{Caruso/LeBorgne:2017} is not translatable to a wider context. The results in \cite{Caruso/LeBorgne:2017} depend heavily on the triviality of the Brauer group of any finite field.

\subsection{Rational functions over finite fields}\label{ratfunc}

Consider the ring \(R = \ffield(t)[X;\sigma,\delta]\), where $\ffield(t)$ is the field of rational functions over the finite field $\ffield=\ffield[q]$. This case was studied in \cite{Giesbrecht/Zhang:2003}, where Algorithm \texttt{Factorization} in \cite[\S 4]{Giesbrecht/Zhang:2003} is proposed to decompose any \(f \in R\). Unfortunately, this paper contains two mistakes, so it cannot be said that a factorization algorithm on \(\ffield[q](t)[X;\sigma,\delta]\) is known. 

The first step in Algorithm \texttt{Factorization} when applied to the pure automorphism case $R = \mathbb{F}_q(t)[x;\sigma]$ consists in the computation of a basis and the structure constants of the Eigenring \(\mathcal{E}(Rf)\) as vector space over the invariant subfield \(K = \ffield[q](t)^\sigma\). It is necessary to provide a computable description of \(K\) for all \(\sigma \in \mathrm{Aut}_{\ffield[q]}(\ffield[q](t))\). In \cite{Giesbrecht/Zhang:2003}, the authors give the corresponding description when \(\sigma\) is a dilation or a shift. The reduction to dilations can be done replacing \(\ffield[q](t)\) to \(\ffield[q^2](t)\), in order to ensure that each matrix has a Jordan normal form. However this replacement can increase the number of factors of the polynomials and provide an incorrect factorization. 

 It is well known that $\sigma$ is determined by the linear fractional transformation $t \mapsto \frac{\sigma_1 t + \sigma_2}{\sigma_3 t + \sigma_4} $,  and \cite[Section 2.2]{Giesbrecht/Zhang:2003} proceeds by identifying  \(\mathrm{Aut}_{\mathbb{F}_q}{\mathbb{F}_q(t)}\) with \(\operatorname{PGL}(2,\mathbb{F}_q)\) via the map \(\Phi:s = \left( \begin{smallmatrix} \sigma_1 & \sigma_2 \\ \sigma_3 & \sigma_4 \end{smallmatrix} \right) \mapsto \frac{\sigma_1 t + \sigma_2}{\sigma_3 t + \sigma_4} \).  If the characteristic polynomial of $s$ is irreducible over $\mathbb{F}_q$,  then \(\sigma\) is reduced to a dilation by means of the Jordan form of $s$ over $\mathbb{F}_{q^2}$ say \(u s u^{-1}\), for some non singular matrix $u$ with entries in $\mathbb{F}
_{q^2}$.  Then the automorphism \(\overline{\sigma} := \tau^{-1} \circ \sigma \circ \tau\) (not $\tau \circ \sigma \circ \tau^{-1}$, since $\Phi$ is already a group anti-isomorphism) of $\mathbb{F}_{q^2}(t)$, where \(\tau\) is the fractional linear transformation corresponding to \(u\), is of dilation type. The automorphism $\tau^{-1}$ extends canonically to an isomorphism of rings \(\mathbb{F}_{q^2}(t)[x; \sigma] \cong \mathbb{F}_{q^2}(t)[\overline{x}; \overline{\sigma}]\). In \cite[Section 2.2]{Giesbrecht/Zhang:2003} it is claimed that the factorizations of $f \in R$ corresponds to factorizations of $\tau^{-1}(f)$ in $\mathbb{F}_{q^2}(t)[\overline{x}; \overline{\sigma}]$. However, this seems not to be always the case, as Example \ref{counterexample} shows.

\begin{example}\label{counterexample}
Let us apply the procedures in \cite{Giesbrecht/Zhang:2003} to the polynomial
\[
f(X) = X^2 + \frac{t^2 + 1}{t} X + (t^2 + t + 1) \in \ffield[2](t)[X;\sigma]
\]
where \(\sigma(t) = \frac{t+1}{t}\). The characteristic polynomial of \(\left( \begin{smallmatrix} 1 & 1 \\ 1 & 0 \end{smallmatrix} \right)\) is irreducible, hence we have to view \(\sigma\) in  \(\Aut[{\ffield[4]}]{\ffield[4](t)}\) and \(f \) in \(\ffield[4](t)[X; \sigma]\). Now we compute the Jordan form:
\[
\left( \begin{matrix} \alpha^2 & 0 \\ 0 & \alpha \end{matrix} \right) = \left( \begin{matrix} \alpha^2 & 1 \\ \alpha & 1 \end{matrix} \right) \left( \begin{matrix} 1 & 1 \\ 1 & 0 \end{matrix} \right) \left( \begin{matrix} 1 & 1 \\ \alpha & \alpha^2 \end{matrix} \right).
\]
As remarked before, \(\tau^{-1}(t) = \frac{t + 1}{\alpha t + \alpha^2}\) is extended naturally to an isomorphism \(\ffield[4](t)[X; \sigma] \cong \ffield[4](t)[\overline{X}; \overline{\sigma}]\), where \(\overline{\sigma}(t) = \alpha t\). So we need to factorize \(\tau^{-1}(f) \in \ffield[4](t)[\overline{X}; \overline{\sigma}]\). Since
\[
\tau^{-1}(f) = \overline{X}^{2} + \left(\frac{\alpha^2 t^{2} + 1}{\alpha^2
t^{2} + \alpha t + 1}\right) \overline{X} + \frac{t}{\alpha^2 t^{2} + \alpha} = \left( \overline{X} + \frac{\alpha t}{\alpha t + \alpha^2} \right) \left( \overline{X} + \frac{\alpha^2}{\alpha t + \alpha^2} \right),
\]
this leads to the following factorization of \(f\),
\[
f = X^2 + \frac{t^2 + 1}{t} X + (t^2 + t + 1) = \left( X + t + \alpha \right) \left( X + t + \alpha^2 \right).
\]
But this is a factorization of \(f\) as Ore polynomial in \(\ffield[4](t)[X;\sigma]\), which is the one that should be obtained by Algorithm \texttt{Factorization} in \cite[pp. 132]{Giesbrecht/Zhang:2003}. Nevertheless this factorization cannot lead to  any factorization of \(f\) as Ore polynomial in \(\ffield[2](t)[X;\sigma]\), since \(f\) is irreducible. In order to check this, a bound of \(f\) can be computed by Algorithms \ref{alg:bound} or \ref{alg:bound2}. Concretely, 
\[
f^* = X^{6} + X^{3} + \frac{t^{6} + t^{5} + t^{3} + t + 1}{t^{4} + t^{2}}.
\]
Since \(\sigma\) has order \(3\), the element \(s = (\sigma^2 + \sigma + 1)(t) = \frac{t^3 + t + 1}{t^2 + t}\) is invariant under \(\sigma\). We have then \(\ffield[2](s) \subseteq K \subseteq \ffield[2](t)\). By \cite[Theorem pp. 197]{vanderWaerden:1949}, it follows that \([\ffield[2](t):\ffield[2](s)] = 3\), hence \([K:\ffield[2](s)] = 1\) and \(K = \ffield[2](s)\). So \(C(\ffield[2](t)[X;\sigma]) = \ffield[2](s)[X^3]\) by \cite[Theorem 2.8]{Lam/Leroy:1988} or \cite[Theorem 1.1.22]{Jacobson:1996}. Now, since
\[
\frac{t^{6} + t^{5} + t^{3} + t + 1}{t^{4} + t^{2}} = \left( \frac{t^3 + t + 1}{t^2 + t} \right)^2 + \frac{t^3 + t + 1}{t^2 + t} + 1,
\]
it follows that 
\[
f^* = (X^3)^2 + (X^3) + s^2 + s + 1 \in \ffield[2](s)[X^3],
\]
which is irreducible. Hence \(f\) is irreducible by Proposition \ref{boundmaxdegree}.
\end{example}

Example \ref{counterexample} explains why Algorithm \texttt{Factorization} in \cite[pp. 132]{Giesbrecht/Zhang:2003} fails if \(\sigma\) is not a dilation nor a shift. 

Nevertheless, this gap can be amended by providing a general method for describing the invariant subfield of $\mathbb{F}(t) = \mathbb{F}_q(t)$ under $\sigma$. Let $\mu$ denote the order of $\sigma$, and let \(K = \ffield(t)^\sigma\). A description of \(K\) appears in \cite{Gutierrez/Sevilla:2006JA} for any finite subgroup \(H \leq \mathrm{Aut}_{\ffield}{\ffield(t)}\). Let us apply its results to our setting, i.e. \(H = \{1, \sigma, \dots, \sigma^{\mu-1} \}\). In this case, \cite[Algorithm 1]{Gutierrez/Sevilla:2006JA} can be written as shown in Algorithm \ref{InvSubfield}.

\begin{algorithm}
\caption{Invariant subfield. \cite{Gutierrez/Sevilla:2006JA}}
\label{InvSubfield}
\begin{algorithmic}
\REQUIRE{$\sigma \in \mathrm{Aut}_{\ffield}{\ffield(t)}$}
\REQUIRE{$e_0, \dots, e_\mu$ the elementary symmetric functions}
\ENSURE{$s \in \ffield(t)$ such that $\ffield(t)^\sigma = \ffield(s)$}
\FOR{$i = 0, \dots, \mu-1$}
\STATE $h_i \gets \sigma^i(t)$
\ENDFOR
\STATE $i \gets 1$
\REPEAT
\STATE $s \gets e_i(h_0, \dots, h_{\mu-1})$
\STATE $i \gets i+1$
\UNTIL{$s \notin \ffield$}
 \RETURN $s$
\end{algorithmic}
\end{algorithm}

Correctness of Algorithm \ref{InvSubfield} is ensured by \cite[Theorem 15]{Gutierrez/Sevilla:2006JA}. It remains to find a procedure to write any \(f \in \ffield(t)^\sigma\) as a rational function in \(s\), where \(\ffield(t)^\sigma = \ffield(s)\), i.e. we want to find \(g \in \ffield(t)\) such that \(f = g(s)\). This is the Functional Decomposition Problem for univariate rational functions. Although there is a large literature in this FDP, for our purposes, we may refer the approach in \cite{Dickerson:1989}, where the coefficients of \(g\) are computed solving the appropriate system of linear equations. The procedure is better understood with an example:

\begin{example}
We present in this example how the polynomial \(f^*\) in Example \ref{counterexample} is written as an element in \(C(\ffield[2](t)[X;\sigma]) = \ffield[2](s)[X^3]\), where \(s = \frac{t^3 + t + 1}{t^2 + t}\). Although we have computed \(s\) directly in Example \ref{counterexample}, \(s\) is the output of Algorithm \ref{InvSubfield}. In order to do so, we need to find \(g \in \ffield[q](t)\) such that 
\[
\frac{t^{6} + t^{5} + t^{3} + t + 1}{t^{4} + t^{2}} = g \left( \frac{t^3 + t + 1}{t^2 + t} \right).
\]
Since \(\deg(f) = 6\) and \(\deg(s) = 3\), it follows that \(\deg(g) = 2\), i.e. \(g = \frac{g_0 + g_1 t + g_2 t^2}{g_3 + g_4 t + g_5 t^2}\). Then 
\[
\begin{split}
\frac{t^{6} + t^{5} + t^{3} + t + 1}{t^{4} + t^{2}} &= \frac{g_0 + g_1 \frac{t^3 + t + 1}{t^2 + t} + g_2 \left( \frac{t^3 + t + 1}{t^2 + t} \right)^2}{g_3 + g_4 \frac{t^3 + t + 1}{t^2 + t} + g_5 \left( \frac{t^3 + t + 1}{t^2 + t} \right)^2} \\
&= \frac{g_0 (t^2 + t)^2 + g_1 (t^3 + t + 1)(t^2 + t) + g_2 (t^3 + t + 1)^2}{g_3 (t^2 + t)^2 + g_4 (t^3 + t + 1)(t^2 + t) + g_5 (t^3 + t + 1)^2} \\
&= \frac{g_2 + g_1 t + (g_0 + g_2) t^2 + g_1 t^3 + (g_0 + g_1) t^4 + g_1 t^5 + g_2 t^6}{g_5 + g_4 t + (g_3 + g_5) t^2 + g_4 t^3 + (g_3 + g_4) t^4 + g_4 t^5 + g_5 t^6},
\end{split}
\]
which leads to the following linear equations
\begin{align*}
g_2 &= 1 \\
g_1 &= 1 \\
g_0 + g_2 &= 0 \\
g_0 + g_1 &= 0 \\
g_5 &= 0 \\
g_4 &= 0 \\
g_3 + g_5 &= 1 \\
g_3 + g_4 &= 1, 
\end{align*}
whose solution is \(g_0 = g_1 = g_2 = g_3 = 1, g_4 = g_5 = 0\), i.e. \(g = 1 + t + t^2\). We then get 
\[
\frac{t^{6} + t^{5} + t^{3} + t + 1}{t^{4} + t^{2}} = g(s) = 1 + \frac{t^3 + t + 1}{t^2 + t} + \left( \frac{t^3 + t + 1}{t^2 + t} \right)^2
\]
as pointed out in Example \ref{counterexample}. We refer to \cite{Alonso/Gutierrez/Recio:1995, Zippel:1991} for other references on this problem.
\end{example}

The second mistake is much deeper. Steps (2) and (3) in Algorithm \texttt{Factorization} in \cite[\S 4]{Giesbrecht/Zhang:2003} are based in the algorithms and procedures in \cite{Ivanyos/Ronyai/Szanto:1994}, where for a given finite-dimensional algebra \(\mathfrak{A}\) over a finite extension of $\ffield[](t)$, the Jacobson radical is computed \cite[Theorem 3.6]{Ivanyos/Ronyai/Szanto:1994} and if \(\mathfrak{A}\) is semisimple, the minimal twosided ideals are also computed, or equivalent a complete set of \textbf{central} idempotents, providing a Wedderburn decomposition of \(\mathfrak{A}\) as a direct sum of simple algebras. However, in \cite[\S 4]{Giesbrecht/Zhang:2003} the authors say that Ivanyos et al. algorithm provides a set of \textbf{primitive} orthogonal idempotents, reporting that \(\mathfrak{A}\) is a division algebra if they do not exist. This is false  because  \(\mathfrak{A}\) could be simple but not a division ring.  So, given \(f \in R = \ffield[q](t)[X;\sigma,\delta]\),  what the iterated application of algorithm \texttt{Factorization} in \cite{Giesbrecht/Zhang:2003}  already computes is a decomposition \(f\) as \(f = g_1 \dots g_r\), where the eigenring of each \(g_i\) is a simple finite-dimensional algebra over a subfield of $\mathbb{F}(t)$. Since each $R/Rg_i$ is then a finite direct sum of isomorphic simple left $R$--modules, we get that $R/Rg_i^*$ is a simple algebra, too.  Thus, this factorization is  in fact the decomposition provided by Proposition \ref{roughdecomposition}. Nevertheless, the use of Proposition \ref{roughdecomposition} has the following advantages. Firstly, our procedure only requires to compute a bound $f^*$ once and all the factors of this ``rough'' decomposition come from the complete factorization of $f^*$ in the commutative polynomial ring $\mathbb{F}(t)^{\sigma}[z]$. 
Secondly, once the ``rough'' decomposition is done, Proposition \ref{boundmaxdegree} gives a halting condition since it can identify some irreducible polynomials. Finally, in order to treat the wild case, i.e. an irreducible bound with no maximal degree, the structure constants of the algebra $(1-e)A(1-e)$ come  from  the ones of $A=R/Rf^*$, which are the same for any polynomial whose bound is $f^*$. We thus avoid the costly computation of the eigenring of the factors for each partial factorization when \texttt{Factorization} from \cite{Giesbrecht/Zhang:2003} is applied.

As far as we know there is not an algorithm to find zero divisors in these simple algebras given their structure constants. So the problem is still open until such an algorithm is published.

\section{Conclusions and future work}
In this paper we give two algorithms for computing the bound of an Ore polynomial in $R = D[X;\sigma,\delta]$, where $D$ is a division ring, $\sigma:D\to D$ an automorphism and $\delta$ a $\sigma$-derivation. Essentially, these algorithms run under the assumption that the ring is finitely generated over its center and, obviously, the data $D$, $\sigma$ and $\delta$ are effective. {Under mild conditions, the center of $R$ is a commutative polynomial ring $K[z]$.  One of the algorithms needs not to know a set of generators of the ring over its center, or even this center. It only requires a set of generators of $R$ as an algebra over its center. In many situations, this number is $2$.} The bound of an Ore polynomial allows us to link the factorization problem
in $D[X;\sigma,\delta]$ with  the factorization into irreducibles in $K[z]$, or, more generally, the factorization of twosided polynomials. In this sense,  a factorization of the bound produces a decomposition of the given polynomial into factors with prime twosided bounds (or irreducible bounds, when the center is $K[z]$).

The next steps can be directed to solve the problem of factorization whenever any bound of the polynomial is irreducible and does not achieve the maximal degree. Actually, this is the problem of factoring out polynomials in the center, as noncommutative polynomials. Nevertheless, solving this problem under the generality of Algorithm \ref{generalfactorization} seems to be  an unrealistic aim. Something more approachable should be the study of some particular cases. For instance, following Subsection \ref{ratfunc}, the reader may consider Ore polynomials over $\ffield[](t)$. It is required an algorithm for finding zero divisors, or just an element whose square is zero, in a simple algebra over a finite extension of $\ffield[](z)$, for some computable rational function $z$.

\section*{Acknowledgments}
We thank T. Recio, J. R. Sendra and L. F. Tabera for their help computing rational points in Example \ref{ex:factquat}. We also thank the anonymous referees for their reports, that lead to many improvements in the presentation of our results.


\begin{thebibliography}{30}
\providecommand{\natexlab}[1]{#1}
\providecommand{\url}[1]{\texttt{#1}}
\expandafter\ifx\csname urlstyle\endcsname\relax
  \providecommand{\doi}[1]{doi: #1}\else
  \providecommand{\doi}{doi: \begingroup \urlstyle{rm}\Url}\fi

\bibitem[Alonso et~al.(1995)Alonso, Gutierrez, and
  Recio]{Alonso/Gutierrez/Recio:1995}
C.~Alonso, J.~Gutierrez, and T.~Recio.
\newblock A rational function decomposition algorithm by near-separated
  polynomials.
\newblock \emph{Journal of Symbolic Computation}, 19\penalty0 (6):\penalty0 527
  -- 544, 1995.
\newblock ISSN 0747-7171.
\newblock \doi{10.1006/jsco.1995.1030}.

\bibitem[Atiyah and Macdonald(1969)]{Atiyah/Macdonald:1969}
M.~F. Atiyah and I.~G. Macdonald.
\newblock \emph{Introduction to commutative algebra}.
\newblock Number 361 in Addison-Wesley Series in Mathematics. Addison-Wesley,
  Reading, MA, 1969.
  
\bibitem[Boucher and Ulmer(2009)]{Boucher/Ulmer:2009}
D.~Boucher, and F.~Ulmer.
\newblock Coding with skew polynomial rings. 
\newblock \emph{Journal of Symbolic Computation}, 44: 1644--1656, 2009.
\newblock \doi{10.1016/j.jsc.2007.11.008}

\bibitem[Boucher and Ulmer(2014)]{Boucher/Ulmer:2014}
D.~Boucher, and F.~Ulmer.
\newblock Linear codes using skew polynomials with automorphism and derivation. 
\newblock \emph{Design, Codes and Cryptography}, 70: 405--431, 2014.
\newblock \doi{10.1007/s10623-012-9704-4}

\bibitem[Boulagouaz and Leroy(2013)]{Boulagouaz/Leroy:2013}
M.~Boulagouaz, and A. Leroy.
\newblock \((\sigma,\delta)\)-Codes.
\newblock \emph{Advances in Mathematics of Communications}, 7(4): 463--474, 2013. 
\newblock \doi{10.3934/amc.2013.7.463}

\bibitem[Bueso et~al.(2003)Bueso, G{\'o}mez-Torrecillas, and
  Verschoren]{Bueso/Gomez/Verschoren:2003}
J.~Bueso, J.~G{\'o}mez-Torrecillas, and A.~Verschoren.
\newblock \emph{Algorithmic Methods in Non-Commutative Algebra: Applications to
  Quantum Groups}.
\newblock Mathematical Modelling: Theory and Applications. Springer, 2003.
\newblock ISBN 9781402014024.
\newblock \doi{10.1007/978-94-017-0285-0}

\bibitem[Cantor and Kaltofen(1991)]{Cantor/Kaltofen:1991}
D.~G. Cantor and E.~Kaltofen.
\newblock On fast multiplication of polynomials over arbitrary algebras.
\newblock \emph{Acta Informatica}, 28\penalty0 (7):\penalty0 693--701, 1991.
\newblock ISSN 0001-5903.
\newblock \doi{10.1007/BF01178683}.

\bibitem[{Caruso} and {Le Borgne}(2012)]{Caruso/LeBorgne:2017}
X.~{Caruso} and J.~{Le Borgne}.
\newblock {A new faster algorithm for factoring skew polynomials over finite fields}.
\newblock \emph{Journal of Symbolic Computation}, 79 (2):\penalty0 411 -- 443, 2017.
\newblock \doi{10.1016/j.jsc.2016.02.016}.

\bibitem[Cauchon(1977)]{Cauchon:1977}
G.~Cauchon.
\newblock \emph{Les T-anneaux et les anneaux \`{a} identi\'{e}s polynomiales noeth\'{e}riens.}
\newblock Th\`{e}se, Orsay, 1977.

\bibitem[Cohn(1971)]{Cohn:1971}
P.~Cohn.
\newblock \emph{Free Rings and Their Relations}.
\newblock L.M.S. Monographs. Acad. Press, 1971.
\newblock ISBN 9780121791506.

\bibitem[Dickerson(1989)]{Dickerson:1989}
M.~T. Dickerson.
\newblock \emph{The functional decomposition of polynomials}.
\newblock PhD thesis, Department of Computer Science, Cornell University,
  Ithaca, NY, 1989.

\bibitem[von~zur Gathen and Gerhard(2003)]{vonzurGathen/Gerhard:2003}
J.~von~zur Gathen and J.~Gerhard.
\newblock \emph{Modern Computer Algebra}.
\newblock Cambridge University Press, New York, NY, USA, 2 edition, 2003.
\newblock ISBN 0521826462.

\bibitem[von~zur Gathen and Panario(2001)]{vonzurGathen/Panario:2001}
J.~von~zur Gathen and D.~Panario.
\newblock Factoring polynomials over finite fields: {A} survey.
\newblock \emph{J. Symb. Comput.}, 31\penalty0 (1/2):\penalty0 3--17, 2001.
\newblock \doi{10.1006/jsco.1999.1002}.

\bibitem[von~zur Gathen and Shoup(1992)]{vonzurGathen/Shoup:1992}
J.~von~zur Gathen and V.~Shoup.
\newblock Computing frobenius maps and factoring polynomials.
\newblock \emph{Computational Complexity}, 2:\penalty0 187--224, 1992.
\newblock \doi{10.1007/BF01272074}.

\bibitem[Giesbrecht(1998)]{Giesbrecht:1998}
M.~Giesbrecht.
\newblock {Factoring in Skew-Polynomial Rings over Finite Fields}.
\newblock \emph{Journal of Symbolic Computation}, 26\penalty0 (4):\penalty0
  463--486, 1998.
\newblock \doi{10.1006/jsco.1998.0224}.

\bibitem[Giesbrecht and Zhang(2003)]{Giesbrecht/Zhang:2003}
M.~Giesbrecht and Y.~Zhang.
\newblock Factoring and decomposing ore polynomials over fq(t).
\newblock In \emph{Proceedings of the 2003 International Symposium on Symbolic
  and Algebraic Computation}, ISSAC '03, pages 127--134, New York, NY, USA,
  2003. ACM.
\newblock ISBN 1-58113-641-2.
\newblock \doi{10.1145/860854.860888}.

\bibitem[Gomez-Torrecillas (2014)]{Gomez:2014}
J. G\'omez-Torrecillas. 
\newblock Basic Module Theory over Non-commutative Rings with Computational Aspects of Operator Algebras. 
\newblock In \emph{Algebraic and Algorithmic Aspects of Differential and Integral Operators}, M. Barkatou, T. Cluzeau, G. Regensburger, M. Rosenkranz, eds. LNCS 8372, pages 23-82, Springer, 2014.

\bibitem[G\'{o}mez-Torrecillas et~al.(2016)G\'{o}mez-Torrecillas, Lobillo, and Navarro]{GLN2016}
J.~G\'{o}mez-Torrecillas, F.~J.~Lobillo, and G.~Navarro.
\newblock A new perspective of cyclicity in convolutional codes. 
\newblock \emph{IEEE Transactions on Information Theory}, 62(5): 2702--2706, 2016.
\newblock \doi{10.1109/TIT.2016.2538264}

\bibitem[Goodearl and Warfield(2004)]{Goodearl/Warfield:2004}
K.~R. Goodearl and R.~B. Warfield, Jr.
\newblock \emph{An introduction to noncommutative {N}oetherian rings},
  volume~61 of \emph{London Mathematical Society Student Texts}.
\newblock Cambridge University Press, Cambridge, second edition, 2004.
\newblock ISBN 0-521-83687-5; 0-521-54537-4.
\newblock \doi{10.1017/CBO9780511841699}.

\bibitem[Gutierrez and Sevilla(2006)]{Gutierrez/Sevilla:2006JA}
J.~Gutierrez and D.~Sevilla.
\newblock Building counterexamples to generalizations for rational functions of
  {Ritt}'s decomposition theorem.
\newblock \emph{Journal of Algebra}, 303\penalty0 (2):\penalty0 655 -- 667,
  2006.
\newblock ISSN 0021-8693.
\newblock \doi{10.1016/j.jalgebra.2006.06.015}.

\bibitem[Hoeven(2002)]{VanDerHoeven:2002}
J.~V.~D. Hoeven.
\newblock Fft-like multiplication of linear differential operators.
\newblock \emph{Journal of Symbolic Computation}, 33\penalty0 (1):\penalty0 123
  -- 127, 2002.
\newblock ISSN 0747-7171.
\newblock \doi{10.1006/jsco.2000.0496}.

\bibitem[Ivanyos et~al.(1994)Ivanyos, R{\'{o}}nyai, and
  Sz{\'{a}}nt{\'{o}}]{Ivanyos/Ronyai/Szanto:1994}
G.~Ivanyos, L.~R{\'{o}}nyai, and {\'{A}}.~Sz{\'{a}}nt{\'{o}}.
\newblock Decomposition of algebras over {$F_q(X_1, ..., X_m)$}.
\newblock \emph{Appl. Algebra Eng. Commun. Comput.}, 5:\penalty0 71--90, 1994.
\newblock \doi{10.1007/BF01438277}.

\bibitem[Jacobson(1943)]{Jacobson:1943}
N.~Jacobson.
\newblock \emph{{The theory of Rings.}}
\newblock {American Mathematical Society}, 531 West 116TH Street, New York
  City, 1943.

\bibitem[Jacobson(1996)]{Jacobson:1996}
N.~Jacobson.
\newblock \emph{{Finite-dimensional division algebras over fields.}}
\newblock {Berlin: Springer}, 1996.
\newblock \doi{10.1007/978-3-642-02429-0}.

\bibitem[Lam and Leroy(1988)]{Lam/Leroy:1988}
T.~Y.~Lam, and A.~Leroy.
\newblock Algebraic conjugacy classes and skew polynomial rings.
\newblock In F. van Oystaeyen, and L. Le~Bruyn (eds) \emph{Perspectives in Ring Theory}. NATO ASI Series (Series C: Mathematical and Physical Sciences), vol 233. Springer, Dordrecht, 1988. 
\newblock \doi{10.1007/978-94-009-2985-2\_15}

\bibitem[Le~Gall(2014)]{LeGall:2014}
F.~Le~Gall.
\newblock Powers of tensors and fast matrix multiplication.
\newblock In \emph{Proceedings of the 39th International Symposium on Symbolic
  and Algebraic Computation}, ISSAC '14, pages 296--303, New York, NY, USA,
  2014. ACM.
\newblock ISBN 978-1-4503-2501-1.
\newblock \doi{10.1145/2608628.2608664}.

\bibitem[Leroy and Ozturk(2011)]{Leroy/Ozturk:2011}
A.~Leroy, and A.~Ozturk.
\newblock Algebraic and \(F\)-Independent Sets in 2-Firs.
\newblock \emph{Communications in Algebra}, 32(5): pp. 1763--1792.
\newblock \doi{10.1081/AGB-120029901}

\bibitem[McConnell et~al.(1987)McConnell, Robson, and
  Small]{McConnell/Robson:1987}
J.~McConnell, J.~Robson, and L.~Small.
\newblock \emph{Noncommutative Noetherian Rings}.
\newblock Wiley series in pure and applied mathematics. John Wiley and Sons,
  1987.

\bibitem[O'Meara(2000)]{OMeara:2000}
O.~T. O'Meara.
\newblock \emph{Introduction to quadratic forms}.
\newblock Classics in mathematics. Springer, Berlin, Heidelberg, Paris, 2000.
\newblock ISBN 3-540-66564-1.
\newblock Reprint of the 1973 edition.

\bibitem[Ore(1933)]{Ore:1933}
O.~Ore.
\newblock Theory of non-commutative polynomials.
\newblock \emph{Annals of Mathematics}, 34\penalty0 (3):\penalty0 pp. 480--508,
  1933.
\newblock ISSN 0003486X.

\bibitem[Pierce(1982)]{Pierce:1982}
R.~Pierce.
\newblock \emph{Associative algebras}.
\newblock Graduate texts in mathematics. Springer-Verlag, 1982.
\newblock ISBN 9780387906935.

\bibitem[R\'{o}nyai(1987)]{Ronyai:1987}
L.~R\'{o}nyai.
\newblock Simple algebras are difficult.
\newblock In \emph{Proceedings of the Nineteenth Annual ACM Symposium on Theory
  of Computing, 25-27 May 1987, New York City, NY, USA}, pages 398--408. ACM,
  1987.

\bibitem[R{\'o}nyai(1988)]{Ronyai:1988}
L.~R{\'o}nyai.
\newblock Zero divisors in quaternion algebras.
\newblock \emph{J. Algorithms}, 9\penalty0 (4):\penalty0 494--506, 1988.

\bibitem[Sch\"{o}nhage(1977)]{Schonhage:1977}
A.~Sch\"{o}nhage.
\newblock Schnelle multiplikation von polynomen \"{u}ber k\"{o}rpern der
  charakteristik 2.
\newblock \emph{Acta Informatica}, 7\penalty0 (4):\penalty0 395--398, 1977.
\newblock ISSN 0001-5903.
\newblock \doi{10.1007/BF00289470}.

\bibitem[Sch\"{o}nhage and Strassen(1971)]{Schonhage/Strassen:1971}
A.~Sch\"{o}nhage and V.~Strassen.
\newblock Schnelle multiplikation gro\ss{}er zahlen.
\newblock \emph{Computing}, 7\penalty0 (3-4):\penalty0 281--292, 1971.
\newblock ISSN 0010-485X.
\newblock \doi{10.1007/BF02242355}.

\bibitem[Stein et~al.(2014)]{sage}
W.~Stein et~al.
\newblock \emph{{S}age {M}athematics {S}oftware ({V}ersion 5.12)}.
\newblock The Sage Development Team, 2014.
\newblock {\tt http://www.sagemath.org}.

\bibitem[van~der Waerden(1949)]{vanderWaerden:1949}
B.~L. van~der Waerden.
\newblock \emph{Modern Algebra}, volume~I.
\newblock Frederick Ungar Publishing Co., 1949.

\bibitem[Zippel(1991)]{Zippel:1991}
R.~Zippel.
\newblock Rational function decomposition.
\newblock In \emph{Proceedings of the 1991 International Symposium on Symbolic
  and Algebraic Computation}, ISSAC '91, pages 1--6, New York, NY, USA, 1991.
  ACM.
\newblock ISBN 0-89791-437-6.
\newblock \doi{10.1145/120694.120695}.

\end{thebibliography}

\end{document}